\documentclass[%
  a4paper,
 onecolumn,
  colorlinks,
]{preprint}

\usepackage[english]{babel}

\usepackage{graphicx}
\usepackage{epstopdf}
\usepackage{subcaption}

\usepackage{amsmath}
\usepackage{amssymb}
\usepackage{amsthm}
\usepackage[ruled,vlined]{algorithm2e}
\usepackage{aliascnt}
\usepackage[nameinlink]{cleveref}

\def\R{\mathbb{R}}
\def\C{\mathbb{C}}

\def\N{\mathbb{N}}
\def\T{\mathcal{T}}
\def\M{\mathcal{M}}

\def\K{\mathcal{K}}

\def\Re{\ensuremath{\text{Re}}}
\def\Im{\ensuremath{\text{Im}}}

\def\d{\ensuremath{\mathrm{d}}}
\newcommand{\imunit}{{\dot{\imath\!\imath}}}

\definecolor{cadmiumgreen}{rgb}{0.0, 0.42, 0.24}

\DeclareMathOperator*{\argmin}{arg\,min}
\newcommand{\bbm}{\begin{bmatrix}}
\newcommand{\ebm}{\end{bmatrix}}
\newcommand{\bb}[1]{\mathbf{#1}}
\newcommand{\id}{\mathrm{id}}
\newcommand{\dom}{\mathcal{D}}

\newtheorem{theorem}{Theorem}

\theoremstyle{definition}
\newtheorem{definition}{Definition}[section]

\newaliascnt{lemma}{theorem}
\newtheorem{lemma}[lemma]{Lemma}
\aliascntresetthe{lemma}

\newaliascnt{corollary}{theorem}
\newtheorem{corollary}[corollary]{Corollary}
\aliascntresetthe{corollary}

\crefname{corollary}{Corollary}{Corollaries}
\Crefname{corollary}{Corollary}{Corollaries}
\crefname{lemma}{Lemma}{Lemmas}
\Crefname{lemma}{Lemma}{Lemmas}

\begin{document}

\title{Wavelet-Based Observables for Koopman Analysis: An Extended Dynamic Mode Decomposition Framework}

\author[$\ast$]{Cankat Tilki}
\affil[$\ast$]{Department of Mathematics, Virginia Tech, Blacksburg, VA - 24061.\authorcr
  \email{cankat@vt.edu}, \orcid{0009-0001-8667-3075}}

\author[$\dagger$]{Serkan G{\" u}{\u g}ercin}
\affil[$\dagger$]{Department of Mathematics, Virginia Tech, Blacksburg, VA - 24061.\authorcr
  \email{gugercin@vt.edu}, \orcid{0000-0003-4564-5999}}

\shorttitle{Wavelet-Based Observables for Koopman Analysis}
\shortauthor{C.Tilki, S.G{\" u}{\u g}ercin}
\shortdate{}
\shortinstitute{}



\abstract{%
  We present an in-depth analysis of the Koopman semigroup via wavelet transform. Towards this goal, we start by introducing the \emph{wavelet-based observables} and show that they are eigenfunctions of the Koopman semigroup when this semigroup is considered over the Banach space of continuous functions on a compact forward-invariant set endowed with the supremum norm. We then construct closed-form expressions of the action of the Koopman semigroup and its resolvent in terms of these observables.
  To approximate the action of Koopman semigroup numerically, we combine Extended Dynamic Mode Decomposition (EDMD) with the proposed wavelet-based observables leading to the Wavelet Dynamic Mode Decomposition via Continuous Wavelet Transform (cWDMD) algorithm. We validate our theoretical results on two numerical examples.
}

\novelty{}

\maketitle


\section{Introduction}\label{sect: Intro}
Traditionally, system modeling is done via the first-principles approach. One determines what laws the system is governed by, and then builds a state space model, i.e., a space $\M$ and a map $\mathcal{T}:\M\to\M$ over that space, determined by the properties of that system. In this case, we describe a system by the map $\mathcal{T}$ as
\begin{equation}\label{eq: intro_1}
    \dot{\bb{x}} = \mathcal{T}(\bb{x}),\quad \bb{x}(0) = \bb{x}_0\in \M.
\end{equation}
This representation is descriptive and intuitive, but it requires knowing the principles that govern the system itself. For simple physical phenomena, this assumption is not so restrictive. However, as systems become more complex, it becomes harder to isolate and model these principles. That restriction motivated the advance of data-driven modeling. Simply put, data-driven modeling uses data acquired from the system to build a model. Generally, this data comes from sensors measuring some aspect of the system of interest. Therefore, a state space representation of the form \cref{eq: intro_1} is not always obvious.
Data-driven modeling has caught substantial attention, especially in research areas involving complex systems, which are hard to model by first-principles approaches. Influenced by the advances in many fields, e.g., statistics, approximation theory, etc., methods have been devised to model a system with reasonable accuracy by using the underlying state space data, see, for example, \cite{Kutz_Wave, DDM_ss_1, DDM_ss_2, DDM_ss_3, DDM_ss_4, DMD_ex_1, DMD_ex_2, DMD_ex_3, DMD_ex_4, DMD_ex_5, Schmid_2010, Rowley_2009, delay_embed, Williams_2015} and the references therein.

One way to approach data-driven modeling is via the Koopman operator theory \cite{Koopman, KoopMezic, KoopMezicPDE, KoopBook}. In this approach, rather than directly encoding the laws into a state space model, one models a system by how the \emph{measurements} change over time. This is done by assuming that the data is ``observed" from a space $\M$ by measurement functions, or \emph{observables}, $\psi:\M\to\C$. The system is then described by a semigroup of linear operators $(\mathcal{K}_{\Delta t})_{\Delta t\ge0}$, called the \emph{Koopman semigroup}. For each $\Delta t\ge0$, the operator $\mathcal{K}_{\Delta t}$ advances observables by $\Delta t$ time units, in the sense that for any $\psi:\M\to\C$,
\begin{equation}\label{eq: Koop_intro}
    \mathcal{K}_{\Delta t}[\psi](\bb{x}_0)=\psi(\bb{x}_{\Delta t}),
\end{equation}
where $\bb{x}_{\Delta t}=\bb{x}(\Delta t)$ denotes the state at time $\Delta t$ corresponding to the initial condition $\bb{x}_0\in\M$.

However, the Koopman operator is defined over an infinite dimensional function space. Hence, approximating it poses additional difficulties. Specifically, the spectrum of that operator may also include a continuum of points which captures some information regarding the associated behavior of the dynamics, and approximating this continuum is difficult by itself. Moreover, computing the spectra of an infinite dimensional operator is prone to spectral pollution. In short, this pollution comes from the spurious eigenvalues that emerge due to the discretization of the infinite dimensional operator. For more discussion on these potential issues, we refer to the reader to, e.g., \cite{Colbrook_2022, Williams_2015} and the references therein.

Even so, we can still compute the best linear map - in the least-squares sense - that approximates the action of $\mathcal{K}_{\Delta t}$ on a finite dimensional subspace. This idea results in the Extended Dynamic Mode Decomposition (EDMD) algorithm. We refer the reader to \cite{Williams_2015, mauroy2020koopman, Zlatko_Schur,Otto_2019,Li_2017,mpEDMD,riggedEDMD,confKoop} for further details, extensions and applications.

This framework can be extended to more general systems where the solution trajectory is not directly accessible, but can instead be observed through an output map. In this case, the state space representation involves a state-to-output map, taking the form
\begin{equation}\label{eq: intro_ds}
    \begin{aligned}
        \dot{\bb{x}} &= \mathcal{T}(\bb{x})\\
        \bb{y} &= g(\bb{x})
    \end{aligned}~,\quad \begin{aligned}
        &\T:\M\to\M,~g:\M\to\R \\
        &\bb{x}(0) = \bb{x}_0\in \mathcal{M} \subseteq \R^n,~\bb{y} \in \R^d,
    \end{aligned}
\end{equation}
where $g$ represents the state-to-output map. In such systems, the main focus is on the output trajectory, which is independent of the specific state space representation. Because of that, relying on a fixed state can be restrictive. To overcome this limitation, some data-driven methods aim to approximate the output map directly from output data, without requiring state measurements. This avenue is especially well investigated for linear time-invariant systems: when $\T(\bb{x}) = \bb{A}\bb{x}$ and $g(\bb{x}) = \bb{C}\bb{x}$ for some matrices $\bb{A}\in \R^{n\times n}$ and $\bb{C}\in \R^{n\times d}$, see, e.g., \cite{finH2, DDM_freq_2, DDM_freq_3, DDM_freq_4, DDM_freq_5, DDM_freq_6, DDM_freq_7, DDM_ss_1} for an incomplete list. 

In this linear setting, the Laplace transform of the output trajectory serves as the main analytical tool. Building on this perspective, recent works have studied Laplace transforms of output trajectories of nonlinear systems and shown that they can be interpreted as evaluations of the Koopman resolvent. For example \cite{Disc_KoopRes} investigates the Koopman operator associated with a discrete-time linear time-invariant system and shows that the $z$-transform of the output trajectory coincides with the Koopman resolvent evaluated on the output function. Moreover, in \cite[Section III.C]{Disc_KoopRes}, authors have developed a DMD-type algorithm that is based on these Koopman resolvents. Similarly \cite{Cont_KoopRes} investigates the Koopman semigroup of a continuous-time linear time-invariant system and shows that the Laplace transform of the output trajectory matches with the Koopman resolvent evaluated on the output function. These results suggest that these methods originally developed for linear time-invariant systems may be extended to systems of the form \cref{eq: intro_ds} through their Koopman formulation. This motivates the development of Koopman-based methods that act directly on the output data without requiring access to the state measurements.

One of such methods is the Wavelet-based Dynamic Mode Decomposition (WDMD) \cite{Krishnan_2023}. WDMD proposes a DMD-based approach that captures the output behavior without relying on the actual state space representation. It achieves this by constructing auxiliary states directly from output data using the Maximum Overlap Discrete Wavelet Transform (MODWT) \cite[Chapter 5]{DWT_alg}, a discretized form of the wavelet transform. In this way, auxiliary states contain only the necessary information regarding the output behavior. Then, WDMD assumes the underlying model becomes a linear model and approximates $\T$ and $g$ in \cref{eq: intro_ds} by solving a linear least squares problem. Motivated by their numerical results, authors have conjectured that this method provides an approximation to the Koopman operator of the underlying system.

Inspired by \cite{Krishnan_2023}, we analyze the Koopman semigroup via wavelet theory, derive \emph{wavelet-based observables} and for a suitable choice of wavelet, the corresponding wavelet-based observables are eigenfunctions of the Koopman semigroup and hence span an invariant subspace, see \cref{thm: KoopEfn}. Moreover, we derive closed-form expressions for the action of the Koopman semigroup and its resolvent on the output function in terms of these observables. Employing these observables in the EDMD framework leads to the Wavelet-based Dynamic Mode Decomposition via Continuous Wavelet Transform (cWDMD) algorithm. The method uses the wavelet transform of the output trajectory directly without access to state data. Hence, the resulting approximation is tailored to the output of interest and is independent of any particular state space realization.

Our main contributions are summarized as follows.
\begin{itemize}
    \item In \Cref{thm: Koop_act_wt}, we prove that the Koopman semigroup's action can be explicitly characterized via the wavelet transform.

    \item In \Cref{thm: KoopEfn}, we show that for a specific choice of wavelets, the wavelet-based observables satisfy an eigenvalue equation. Hence, subspace spanned by any finite subset of these observables is invariant under the action of the Koopman semigroup.
    
    \item In \Cref{lem: Kdt_act_Morl,thm: Koop_res_Morl} we show that in this case we have a simple, closed-form expression for the action of the Koopman semigroup and its resolvent evaluated at the output function.
    
    \item Based on these theoretical results, we develop the Wavelet-based Dynamic Mode Decomposition via Continuous Wavelet Transform (cWDMD) method (\Cref{alg: cWDMD}). We then demonstrate the performance of this algorithm on two numerical examples.
\end{itemize}

The paper is organized as follows. 
In \Cref{sect: Prelim}, we introduce the necessary background: we briefly review Koopman operator theory, its connection to Extended Dynamic Mode Decomposition (EDMD), and the wavelet transform. The main theoretical results are presented in \Cref{sect: WaveletKoop}. We start by analyzing the Koopman operator via the wavelet transform in \Cref{ss: KoopAn}. Specifically, in \Cref{thm: Koop_act_wt}, we characterize the action of the Koopman operator using the wavelet transform. In \Cref{ss: wave_obs} we motivate and then introduce the wavelet-based observables. We also provide a closed-form expression for the action of the Koopman semigroup on these observables in \Cref{lem: Koop_wt_obs}. In \Cref{ss: wbobs_Morl}, we show that for a suitable choice of wavelet, the corresponding wavelet-based observables are eigenfunctions of the Koopman semigroup, see \cref{thm: KoopEfn}. This then implies that any finite collection of those observables span an invariant subspace of the Koopman semigroup. Lastly in \Cref{ss: Koop_wave}, we visit our analytic expression in \cref{thm: Koop_act_wt} and derive simple closed-form expressions for the action of the Koopman semigroup in \Cref{lem: Kdt_act_Morl} and the Koopman resolvent in \Cref{thm: Koop_res_Morl}.
In \Cref{sect: Implementation} we discuss how to compute these wavelet-based observables in practice leading to the cWDMD algorithm. In \cref{ss: EDMDImpl} we first review the EDMD algorithm; and then adapt it for our wavelet-based observables and propose \cref{alg: cWDMD} in \Cref{ss: cWDMD}. Finally in \cref{ss: NumLTI}, we demonstrate the cWDMD algorithm on two numerical examples, first a linear time-invariant dynamical system and then the Lorenz 63 system with parameters that make the system chaotic.  
Conclusions are given in \Cref{sect: conc}. 

\section{Mathematical preliminaries}\label{sect: Prelim}

\subsection{The Koopman operator and its approximation}\label{ss: KOT}
In this section, we first provide a brief introduction to the Koopman semigroup. Then, we briefly recall how EDMD can be perceived as an algorithm approximating the underlying Koopman operator for discrete-time systems. We refer the reader to \cite{Williams_2015} for more information.  

We first present the discrete-time formulation where the Koopman operator is defined most directly, and then move to the continuous-time formulation. A discrete-time system is given by
\begin{equation} 
    \bb{z}_{k+1} = \T(\bb{z}_k),\quad\bb{z}\in \M \subset \R^n,~\T: \M \to \M. \label{eq: sysdiscrete}
\end{equation}
Unlike the setup in \cref{eq: intro_1}, here the map $\T$ is defined to describe the next timestep rather than the time derivative. As a result, the trajectory consists of a discrete sequence of points. In this setting, the Koopman operator becomes the composition operator with the dynamics. Hence, it represents the one-step evolution of the system. 
\begin{definition}[Discrete-time Koopman Operator \cite{Williams_2015}]\label{def: dKoop}
    Given a discrete-time dynamical system of the form \eqref{eq: sysdiscrete}, and a space of scalar valued functions $\mathcal{F}$ over the state space $\M$,
    the Koopman operator $\mathcal{K}$ is a possibly unbounded linear operator on $\mathcal{F}$ defined on its domain $\mathcal{D}(\mathcal{K})\subseteq \mathcal{F}$ as
    \begin{equation*}
        \K[f] = f\circ \T,~ \text{for } f\in \mathcal{D}(\K),
    \end{equation*}
   where $f:\M\to\C$ is a scalar valued function on $\M$. 
    Hence, the Koopman operator carries dynamics over the function space by precomposing the scalar valued functions with the dynamics map $\T$.
\end{definition}
In many applications, such discrete-time systems arise from time discretizations of continuous-time systems. Accordingly, the discrete trajectory approximates the continuous-time trajectory at selected time instances. With an appropriate discretization strategy, as the time step, $\Delta t$, approaches zero the discrete-time system will better approximate the underlying continuous-time system. This suggests that we have a family of discrete-time dynamical systems that converges to the continuous-time system. Hence, as $\Delta t \to 0$, one obtains the continuous-time Koopman operator. We formalize this idea as follows.
\begin{definition}[Continuous-time Koopman semigroup and its generator \cite{RowleyReview}]\label{def: cKoop} 
    Given a continuous-time system of the form
    \begin{equation*}\tag{\ref{eq: intro_ds}}
        \dot{\bb{x}} = \T(\bb{x}),~~\bb{x}(0) = \bb{x}_0\in \M,\quad\bb{x}\in \M \subset \R^n,~\T: \M \to \M,
    \end{equation*}
    and a space of scalar valued functions $\mathcal{F}$ over the state space $\M$, let $\Phi(t,\bb{x}_0)$ be the flow of \cref{eq: intro_ds}. Then for every $\Delta t > 0$ precomposing with $\Phi(\Delta t, ~\cdot~)$ defines an operator given as
    \begin{equation} \label{eq: Kdelta}
        \mathcal{K}_{\Delta t}[f] = f\circ\Phi(\Delta t,~\cdot~),~\text{for } f\in \mathcal{F}.
    \end{equation}
    This defines a semigroup of operators over scalar functions. The \emph{generator} $\K$ for the semigroup \cref{eq: Kdelta} is then defined as
    \begin{equation}  \label{eq:Kcont}
        \K[f] = \lim_{\Delta t\downarrow 0} \dfrac{\K_{\Delta t}[f] - f}{\Delta t}.
    \end{equation}
\end{definition}
The generator $\K$ serves as the continuous-time analog of the Koopman operator defined in \cref{def: dKoop}.
Because we will work mostly with vector valued functions, we extend these definitions accordingly. For a simpler notation, given a vector valued function $f = \bbm f_1 & f_2 & \cdots & f_M \ebm^\top$ where $f_j:\mathcal{M}\to \mathbb{C}$ for $j=1,\ldots,M$, we will often abuse the notation of \Cref{def: dKoop,def: cKoop} to write 
\begin{equation*}
    \mathcal{K}[f] := \bbm \mathcal{K}[f_1] & \mathcal{K}[f_2] & \cdots & \mathcal{K}[f_M]\ebm^\top \quad \textrm{for all } f_j:\mathcal{M}\to \mathbb{C},
\end{equation*}
where $(\cdot)^\top$ denotes the transpose operation, i.e.,
\begin{equation*}
    \bbm f_1 & f_2 & \cdots & f_M \ebm^\top = \bbm f_1 \\ f_2 \\ \vdots \\ f_M \ebm.
\end{equation*}
When analyzing this semigroup \cref{eq: Kdelta}, one usually needs more structure. In our case we will assume that this semigroup is $C_0$ (or strongly continuous) over a Banach space $(X,\|\cdot\|_X)$ \cite[Definition 5.1]{EngelNagel}. Specifically, a $C_0$-semigroup is defined to be a semigroup that also satisfies
\begin{equation*}
    \|\K_{\Delta t}[f] - f\|_X \to 0 \textrm{ as } \Delta t \to 0
\end{equation*}
for all $f \in \dom(\K)$. Later in \cref{thm: KoopC0ode}, we will justify this assumption by showing that when $\T$ is Lipschitz continuous and $g$ is continuous, the Koopman semigroup of the ODE \cref{eq: out_DS} is $C_0$ on the Banach space of continuous functions endowed with sup norm $(C(\Omega),\|\cdot\|_{\infty})$ over a compact forward-invariant subset $\Omega \subseteq \M$.

The Koopman operator enables us to model a highly nonlinear dynamical system as a linear but an infinite dimensional one. Specifically, for the discrete-time case, we have
\begin{equation}  \label{eq:discTtoK}
    \bb{z}_{k+1} = \T(\bb{z}_k),~\bb{z} \in \M\subset \R^n \quad\longleftrightarrow \quad f_{k+1} = \K[f_k],~f\in X.
\end{equation}
On the left-hand side of~\eqref{eq:discTtoK}, we have a nonlinear map $\T$ acting on a finite dimensional state space $\M$, while on the right-hand side, the Koopman operator $\K$ is linear, but acts on an infinite dimensional space of functions. This distinction highlights a key tradeoff: instead of directly approximating the nonlinear map $\T$, one can model the dynamical system by approximating the linear, yet infinite dimensional, Koopman operator $\K$.

One method to approximate the action of the Koopman operator is the Extended Dynamic Mode Decomposition (EDMD) \cite{Williams_2015}. The EDMD method approximates the action of the Koopman operator projected on a finite dimensional subspace in a data-driven fashion. Specifically, one picks a finite set of scalar valued functions $\psi_{j}:\M\to \C$ for $j=1,\dots, J$ and state snapshots $\bb{z}_k$ for $k=1,\dots, N$. Because the state data $\bb{z}_k$ are evaluated through $\psi_{j}$, these are called \emph{observables}. 

For discrete-time systems, after choosing observables $\psi_{j}$ and gathering the observed data $\psi_{j}(\bb{z}_k)$ for $k=0,\dots,N$ and $j = 1,\dots,J$, one proceeds by solving the minimization problem
\begin{equation}\label{eq: EDMD_min}
    \hat{\bb{K}} = \argmin_{\bb{K}\in \C^{J\times J}} \sum_{k=0}^{N-1}\|\psi(\bb{z}_{k+1})-\bb{K}\psi(\bb{z}_k)\|_2^2,\quad\psi = \bbm\psi_1 & \psi_2 & \cdots & \psi_J \ebm^\top.
\end{equation}
Then, the matrix $\hat{\bb{K}}\in \C^{J\times J}$ approximates the evolution of the observables, i.e.,
\begin{equation*}
    \psi(\bb{z}_{k+1}) \approx \hat{\bb{K}}\psi(\bb{z}_k).
\end{equation*}
One can apply the same procedure for continuous-time systems, but has to post-process the matrix $\hat{\bb{K}}$ accordingly, see, e.g., \cite{gEDMD}.

The relation between EDMD and the Koopman operator can be observed by the action of the Koopman operator on $\psi$, see \cite{Williams_2015} for a detailed proof. Specifically, for a set of observables $\psi_{j}:\M\to\C$, the EDMD algorithm applied to a discrete dynamical system amounts to approximating the action of the Koopman operator on the subspace spanned by $\psi_{j}$, denoted as $\mathcal{F}_J\subset X$; that is,
\begin{equation}\label{eq: EDMD_Koop_approx}
    \mathcal{K}[f](\bb{x}_i) \approx \psi(\bb{x}_i)^*\hat{\bb{K}}^*\bb{a}_f,\quad f\in \mathcal{F}_J~\text{and}~\bb{a}_f\in \C^J~\text{such that}~ f = \psi^*\bb{a}_f,
\end{equation}
where $\bb{K}^*$ denotes the Hermitian adjoint of $\bb{K}$.

This approximation fails to be exact for two reasons: Invariance of the subspace $\mathcal{F}_J$ and data informativity \cite{Williams_2015}. In short, the subspace $\mathcal{F}_J$ is invariant when 
\begin{equation*}
    \K[f] \in \mathcal{F}_J~ \text{for any } f\in \mathcal{F}_J.
\end{equation*}
This assumption is needed for the approximation \cref{eq: EDMD_Koop_approx} since a matrix representation is possible only if we can represent $\mathcal{K}[f]$ as $\K[f] = \psi^*\widetilde{\bb{a}}_f$ for some $\widetilde{\bb{a}}_f\in \C^M$, which is equivalent to saying that $f$ is invariant under $\K$.
In this paper, we focus on the invariance property and develop a theoretical formulation for constructing a set of observables that satisfy it. We do not focus on the data informativity aspect and assume that we have no control over the data sampling process.

Because of the invariance assumption, the choice of observables $\psi$ becomes crucial and depends mainly on the underlying dynamical system. Many choices for observables have been investigated, including polynomial bases, kernel-based methods, neural networks, and other approaches; see, e.g.~\cite{Schmid_2023, RowleyReview, KutzReview}.
These ideas have demonstrated success across various application domains. However, each comes with limitations. For instance, while a monomial-based set of observables exhibit good local convergence, they may fail to yield a closed system, potentially leading to divergence \cite{Schmid_2023}. Overall, the choice of $\psi$ is still an open question for general systems. See the review papers \cite{Schmid_2023, KutzReview} and references therein for a more detailed discussion about the choice of observables.

In summary, choosing the right $\psi$ is not an easy task, and to our knowledge, no $\psi$ is proven to provide an invariant subspace for non-domain-specific problems. One of our goals for this article is to show that using the wavelet transform, one can construct a set of observables for which the span of $\{\psi_i\}_{i=1}^J$ is invariant under the action of the corresponding Koopman operator; see \Cref{thm: KoopEfn}.

\subsection{Wavelets and wavelet transform}\label{ss: Wavelets}
In this section, we will follow \cite{Dau_92,Mallat} and briefly summarize the wavelets and wavelet transform. Given $h\in L_2(\R)$, the wavelet transform analyzes $h$ by taking inner products with a family of shifted and scaled wavelets $\Gamma_{\sigma,\tau}(t)$, thereby describing how the oscillatory content of the signal changes as a function of time and scale. This is analogous to the Fourier transform, but instead of complex exponentials $e^{i\omega t}$, the wavelet transform uses localized analyzing functions, known as wavelets.

The wavelets $\Gamma_{\sigma,\tau}$ are populated from $\Gamma$; an element of $L_1(\R)\cap L_2(\R)$. One can then define wavelets $\Gamma_{\sigma,\tau}$ as shifts and scales of $\Gamma$ via
\begin{equation}  \label{eq: mothersc}
    \Gamma_{\sigma,\tau}(t) = |\sigma|^{-1}\Gamma\left(\dfrac{t-\tau}{\sigma}\right),\quad \text{for any } \sigma\in\R \setminus\{0\},~\tau\in \R.
\end{equation}
Then for any scale $\sigma\in\R\setminus\{0\}$ and shift $\tau\in \R$ one can define the wavelet transform $\mathcal{W}$ as
\begin{equation}\label{eq: WT}
    \mathcal{W}[h](\sigma,\tau) = \int_{t=-\infty}^\infty h(t)\overline{\Gamma_{\sigma,\tau}(t)} \d t,
\end{equation}
where $\overline{z}$ denotes the complex conjugate of $z\in \C$. Assuming that the $\Gamma$ has zero mean, i.e.,
\begin{equation}\label{eq: mothermean}
    \int_{t = -\infty}^\infty\Gamma(t) \d t = 0, 
\end{equation}
the inverse wavelet transform \cite[Proposition 2.4.1]{Dau_92} is given by 
\begin{equation}\label{eq: iCWT_L2}
    h = \frac{1}{C_{\Gamma}}\int_{\sigma\in \mathbb{R}} \int_{\tau\in \mathbb{R}} \mathcal{W}[h](\sigma,\tau)\Gamma_{\sigma,\tau}\frac{\d\tau \d \sigma}{|\sigma|},~ \textrm{with }C_\Gamma = 2\pi\int_{-\infty}^\infty \frac{|\hat{\Gamma}(\bb{\theta})|^2}{|\theta|}\d\theta,
\end{equation}
where $\hat{\Gamma}$ stands for the Fourier transform~\cite{YosidaFA} of $\Gamma$ defined as
\begin{equation}\label{eq: fourier_transform}
    \hat{\Gamma}(\omega) = \dfrac{1}{\sqrt{2 \pi}}\int_{t\in\R}\Gamma(t)e^{-i\omega t} \d t,~\omega\in \R.
\end{equation}
Note that $\mathcal{W}[h](\sigma,\tau)\Gamma_{\sigma,\tau}$ is not well-defined for $\sigma = 0$. Hence we consider this integral as a limit, see \cite[Equation 2.4.5]{Dau_92}. \Cref{eq: iCWT_L2} should be understood in the $L_2$ sense. One can also show that this equality holds in the pointwise sense even when $h$ is not necessarily in $L_2(\R)$ but only bounded and weakly oscillating around zero \cite[Theorem 2.2]{ICWT_pw}, i.e.,
\begin{equation}\label{eq: wosc0}
    \lim_{r\to \infty}\dfrac{1}{2 r}\int_{t-r}^{t+r}h(\tau) \d \tau = 0\mbox{ uniformly in }t.
\end{equation}
However in this case, one needs to assume that $\Gamma$ decays fast enough so that
\begin{equation*}
    \log(2+|t|)\Gamma(t)\in L_1(\R).
\end{equation*}
Then for a fixed $t \in \R$ we have
\begin{equation}\label{eq: iCWT_pw0}
    h(t) = \frac{1}{C_{\Gamma}}\int_{\sigma\in \mathbb{R}} \int_{\tau\in \mathbb{R}} \mathcal{W}[h](\sigma,\tau)\Gamma_{\sigma,\tau}(t)\frac{\d\tau \d \sigma}{|\sigma|},~ \textrm{with }C_\Gamma = 2\pi\int_{-\infty}^\infty \frac{|\hat{\Gamma}(\bb{\theta})|^2}{|\theta|}\d\theta.
\end{equation}

One important property of the wavelet transform is the translation covariance: a translation of the input signal induces the corresponding translation of the wavelet coefficients in the shift variable. Specifically, given a signal $h$ we expect that a shift $\Delta t$ in $h$ results in the same wavelet transform, but shifted by $\Delta t$. Thus, we have
\begin{equation}\label{eq: shift_inv}
    \mathcal{W}[h_+](\sigma,\tau) = \mathcal{W}[h](\sigma,\tau+\Delta t),\quad \text{for }h_+(t):= h(t+\Delta t).
\end{equation}
This follows from an appropriate change of variables in the integral in \cref{eq: WT}.

\section{Wavelets and Koopman theory}\label{sect: WaveletKoop}
In this section, we analyze the interplay between the Koopman operator and the wavelet transform of the output trajectory. For this, we will derive a closed form expression for the action of the Koopman operator using the wavelet transform at \cref{thm: Koop_act_wt}.

\subsection{Analyzing the Koopman operator using wavelet transform}\label{ss: KoopAn}
The main result of \Cref{ss: KoopAn} is \Cref{thm: Koop_act_wt}, which states that we can write the action of the Koopman operator on a continuous function in closed form by employing the wavelet transform. This analytic expression for the Koopman operator will then be used in \Cref{ss: wave_obs} to motivate the wavelet-based observables $\psi_{\sigma}$, see \cref{eq: wave_obs}.

Consider a continuous-time dynamical system with a scalar output
\begin{equation}\label{eq: out_DS}
    \begin{aligned}
        \dot{\bb{x}} &= \mathcal{T}(\bb{x})\\
        y &= g(\bb{x})
    \end{aligned}~,\quad \left\{\begin{aligned}
        &\bb{x}(0) = \bb{x}_0\in \M \subseteq \R^n \\ 
        &y \in \R.
    \end{aligned}\right.
\end{equation}
Let $\T$ be globally Lipschitz and $g$ be bounded and continuous on $\M$. Then, by the Picard-Lindel{\"o}f theorem \cite{PL_cite}, for every initial condition $\bb{x}_0$ there exists a global solution. Denote the corresponding solution map as $\Phi(t,\bb{x}_0)$ for the initial condition $\bb{x}_0\in \M$ and time $t\in\R$. Thus, the state trajectory corresponding to the initial condition $\bb{x}_0\in\M\subset \R^n$ at time $t\in\R$ is $\Phi(t,\bb{x}_0)$. Then we can show that the associated Koopman semigroup $(\K_{\Delta t})_{t\geq 0}$ is a $C_0$-semigroup on $(C(\Omega),\|\cdot\|_\infty)$ for a compact, forward-invariant subset $\Omega\subseteq \M$.
\begin{theorem}\label{thm: KoopC0ode}
    Consider the dynamical system~\cref{eq: out_DS}. Let $\T$ be globally Lipschitz and $\Omega\subseteq \M$ be a compact, forward-invariant subset of $\M$. Then the Koopman semigroup associated with \cref{eq: out_DS} is a $C_0$-semigroup on $(C(\Omega),\|\cdot\|_\infty)$.
\end{theorem}
\begin{proof}
    Because $\T$ is Lipschitz continuous, the solution map $\Phi(t,\bb{x})$ is continuous in $t\in \R$ and $\bb{x}\in \Omega$. Then, for any fixed $T\in \R_+$, since both $[0,T]$ and $\Omega$ are compact, $\Phi$ becomes uniformly continuous over $[0,T]\times \Omega$. Similarly, any $f\in C(\Omega)$ is continuous on a compact set $\Omega$, hence it is uniformly continuous over $\Omega$. Define the modulus of continuity \cite[Chapter 1.1]{Rivlin} $\omega_f:\R_+ \to \R_+$ for any $f\in C(\Omega)$ as
\begin{equation}\label{eq: defmodcont}
        \omega_f(r) = \sup_{\|\bb{x}-\bb{y}\| < r} |f(\bb{x})-f(\bb{y})|.
    \end{equation}
    From the definition it follows that for any $\bb{x},\bb{y}\in \Omega$
    \begin{equation}\label{eq: modcontbd}
        |f(\bb{x}) - f(\bb{y})| \leq \omega_f(\|\bb{x}- \bb{y}\|).
    \end{equation}
    Moreover, since $f$ is uniformly continuous on $\Omega$ we have that $\omega_f(r)\to 0$ as $r\to 0$ \cite[Lemma 1.2]{Rivlin}. We want to show $\|\K_{\Delta t}[f] -f\|_\infty\to 0$ as $\Delta t\to 0^+$. Using the \cref{def: cKoop}, this can be rewritten as
    \begin{equation}\label{eq: KoopC0ode}
        \|\K_{\Delta t}[f] -f\|_\infty = \sup_{\bb{x}\in \Omega} |f(\Phi(\Delta t,\bb{x})) - f(\bb{x})|.
    \end{equation}
    Hence we want to show
    \begin{equation}\label{eq: KoopC0prop}
        \sup_{\bb{x}\in \Omega} |f(\Phi(\Delta t,\bb{x})) - f(\bb{x})| \to 0 \quad \mbox{as}~\Delta t \to 0.
    \end{equation}
    Consider $|f(\Phi(\Delta t,\bb{x})) - f(\bb{x})|$. Using \cref{eq: modcontbd} we can bound this as
    \begin{equation}\label{eq: Koopmcbd}
        |f(\Phi(\Delta t,\bb{x})) - f(\bb{x})| \leq \omega_f(\|\Phi(\Delta t,\bb{x})- \bb{x}\|).
    \end{equation}
    Taking the supremum of both sides of \cref{eq: Koopmcbd} we have
    \begin{equation}\label{eq: supmodcont}
        \sup_{\bb{x}\in \Omega}|f(\Phi(\Delta t,\bb{x})) - f(\bb{x})| \leq \sup_{\bb{x}\in \Omega}\omega_f(\|\Phi(\Delta t,\bb{x})- \bb{x}\|).
    \end{equation}
    Consider the right hand side of \cref{eq: supmodcont}. Let $r_{\Delta t}\in \R_+$ be defined as
    \begin{equation*}
        r_{\Delta t} := \sup_{\bb{x}\in \Omega}\|\Phi(\Delta t,\bb{x})- \bb{x}\|.
    \end{equation*}
    Note that $r_{\Delta t} < \infty$ since $\Phi$ is continuous in $\bb{x}\in \Omega$ and $\Omega$ is compact. Then because $\omega_f(r)$ is monotonically decreasing in $r$, for any $\bb{x}\in \Omega$ we have
    \begin{equation}\label{eq: mcbdmax}
        \omega_f(\|\Phi(\Delta t,\bb{x}))- \bb{x}\|) \leq \omega_f(r_{\Delta t}).
    \end{equation}
    Taking the supremum of \cref{eq: mcbdmax} with respect to $\Omega$, we get 
    \begin{equation*}
        \sup_{\bb{x}\in \Omega}\omega_f(\|\Phi(\Delta t,\bb{x})- \bb{x}\|) \leq \omega_f(r_{\Delta t}) = \omega_f\left(\sup_{\bb{x}\in \Omega}\|\Phi(\Delta t,\bb{x})- \bb{x}\|\right).
    \end{equation*}
    Then together with \cref{eq: supmodcont} we have
    \begin{equation}\label{eq: modcontlast}
        \sup_{\bb{x}\in \Omega}|f(\Phi(\Delta t,\bb{x})) - f(\bb{x})| \leq \omega_f\left(\sup_{\bb{x}\in \Omega}\|\Phi(\Delta t,\bb{x})- \bb{x}\|\right).
    \end{equation}
    Because $\Phi$ is uniformly continuous in both variables over $[0,T]\times \Omega$, we get 
    \begin{equation*}
        \lim_{\Delta t\downarrow 0}\sup_{\bb{x}\in \Omega}\|\Phi(\Delta t,\bb{x})- \bb{x}\| = 0.
    \end{equation*}
    Since $\omega_f(r)\to 0 $ as $r\to 0$ for $f\in C(\Omega)$, we have that
    \begin{equation*}
        \lim_{\Delta t \downarrow 0}\omega_f\left(\sup_{\bb{x}\in \Omega}\|\Phi(\Delta t,\bb{x})- \bb{x}\|\right) = 0.
    \end{equation*}
    Then \cref{eq: modcontlast} suggests that
    \begin{equation}\label{eq: Koopgenmcbd}
        \lim_{\Delta t\downarrow 0}\sup_{\bb{x}\in \Omega}|f(\Phi(\Delta t,\bb{x})) - f(\bb{x})| \leq \lim_{\Delta t\downarrow 0}\omega_f\left(\sup_{\bb{x}\in \Omega}\|\Phi(\Delta t,\bb{x})- \bb{x}\|\right) = 0.
    \end{equation}
    Since left hand side of \cref{eq: Koopgenmcbd} needs to be nonnegative, we have
    \begin{equation*}
        \lim_{\Delta t\downarrow 0}\sup_{\bb{x}\in \Omega}|f(\Phi(\Delta t,\bb{x})) - f(\bb{x})| = 0.
    \end{equation*}
    This together with \cref{eq: KoopC0ode} yields
    \begin{equation*}
        \lim_{\Delta t\downarrow 0}\sup_{\bb{x}\in \Omega}|f(\Phi(\Delta t,\bb{x})) - f(\bb{x})| = \lim_{\Delta t\downarrow 0}\|\K_{\Delta t}[f]-f\|_\infty = 0.
    \end{equation*}
    This shows that \cref{eq: KoopC0prop} holds and hence $(\K_{\Delta t})_{t\geq 0}$ is a $C_0$-semigroup on $(C(\Omega),\|\cdot\|_\infty)$.
\end{proof}
Our goal is to capture the action of the Koopman generator $\K$ associated with the dynamical system \cref{eq: out_DS} from the output trajectory $y(t,\bb{x}_0):= g\circ\Phi(t,\bb{x}_0)$. As we show next, in this setting, the action of the Koopman operator admits a closed-form expression. 
\begin{theorem}[Action of the Koopman semigroup $\mathcal{K}_{\Delta t}$ on $C(\Omega)$]\label{thm: Koop_act_wt}
    Consider the dynamical system~\cref{eq: out_DS}. Let $\T$ be globally Lipschitz and $g$ be bounded and continuous. Pick a point $\bb{x}_0\in \Omega$ and assume that $y(\cdot,\bb{x}_0)$ satisfies \cref{eq: wosc0} and $\log(2+|t|)\Gamma(t)\in L_1(\R)$. Then the action of an element of the semigroup~\eqref{eq: Kdelta} on $g$ at point $\bb{x}_0\in \M$ is given by
    \begin{equation}\label{eq: Kdt_WT}
        \mathcal{K}_{\Delta t}[g](\bb{x})
        =\frac{1}{C_{\Gamma}}\int_{\sigma\in \R} \int_{\tau\in \R} \mathcal{W}[y(\cdot,\bb{x})](\sigma,\tau+\Delta t) \Gamma_{\sigma,\tau}(0)~\d\tau \dfrac{\d \sigma}{|\sigma|},~\text{for any } \Delta t >0,
    \end{equation}
    where $\K_{\Delta t}$, $\Gamma_{\sigma,\tau}(t)$, $\mathcal{W}[y(\cdot,\bb{x}_0)](\sigma,\tau)$, and $C_{\Gamma}$ are as defined in \cref{eq: Kdelta,eq: WT,eq: mothersc,eq: iCWT_pw0}, respectively. 
\end{theorem}

\begin{proof}
    Fix the time-shift $\Delta t >0$. Then applying $\mathcal{K}_{\Delta t}$ to $g$ yields
    \begin{equation*}
         \mathcal{K}_{\Delta t}[g](\bb{x})= g\circ \Phi(\Delta t,\bb{x}).
    \end{equation*}
    Since $g$ is continuous, $g\circ\Phi(\cdot,\bb{x}) = y(\cdot,\bb{x})$ is continuous. Moreover, since $g$ is bounded, $y(\cdot,\bb{x})$ is bounded for any $\bb{x}$. So assuming $\log(2+|t|)\Gamma(t)\in L_1(\R)$, the inverse wavelet transform \cref{eq: iCWT_pw0} holds pointwise. Then we obtain
    \begin{equation}\label{eq: h_+WT}
        g\circ \Phi(t,\bb{x}) = y(t,\bb{x}) = \frac{1}{C_{\Gamma}}\int_{\sigma\in \R} \int_{\tau\in \R} \mathcal{W}[y(\cdot,\bb{x})](\sigma,\tau)\Gamma_{\sigma,\tau}(t)~\d\tau \dfrac{\d \sigma}{|\sigma|}.
    \end{equation}
    Define $y_{\Delta t}(t,\bb{x}) := y(t+\Delta t,\bb{x})$. Then, by \cref{eq: h_+WT}, we can rewrite $\mathcal{K}_{\Delta t}[g](\bb{x})$ as
    \begin{equation}\label{eq: h_+wave}
        \mathcal{K}_{\Delta t}[g](\bb{x}) = y_{\Delta t}(0,\bb{x}) = \frac{1}{C_{\Gamma}}\int_{\sigma\in \R} \int_{\tau\in \R} \mathcal{W}[y(\cdot+\Delta t,\bb{x})](\sigma,\tau)\Gamma_{\sigma,\tau}(0)~\d\tau \dfrac{\d \sigma}{|\sigma|}.
    \end{equation}
    By translation covariance, (see \cref{eq: shift_inv}), we obtain
    \begin{equation}\label{eq: shift_h}
    \begin{aligned}
        \frac{1}{C_{\Gamma}}\int_{\sigma\in \R} \int_{\tau\in \R} \mathcal{W}&[y(\cdot+\Delta t,\bb{x})](\sigma,\tau) \Gamma_{\sigma,\tau}(0)\d\tau \dfrac{\d \sigma}{|\sigma|} = \\
        &=
        \frac{1}{C_{\Gamma}}\int_{\sigma\in \R} \int_{\tau\in \R} \mathcal{W}[y(\cdot,\bb{x})](\sigma,\tau+\Delta t) \Gamma_{\sigma,\tau}(0)~\d\tau \dfrac{\d \sigma}{|\sigma|}.
    \end{aligned}
    \end{equation}
        Equations \cref{eq: h_+wave} and \cref{eq: shift_h} reveal that we can express $\mathcal{K}_{\Delta t}[g](\bb{x})$ as
    \begin{equation*}
    \begin{aligned}
        \mathcal{K}_{\Delta t}[g](\bb{x}) &= \frac{1}{C_{\Gamma}}\int_{\sigma\in \R} \int_{\tau\in \R} \mathcal{W}[y(\cdot+\Delta t,\bb{x})](\sigma,\tau) \Gamma_{\sigma,\tau}(0)\d\tau \dfrac{\d \sigma}{|\sigma|} \\&=
        \frac{1}{C_{\Gamma}}\int_{\sigma\in \R} \int_{\tau\in \R} \mathcal{W}[y(\cdot,\bb{x})](\sigma,\tau+\Delta t) \Gamma_{\sigma,\tau}(0)\d\tau \dfrac{\d \sigma}{|\sigma|}.
    \end{aligned}
    \end{equation*}
    hence proving \cref{eq: Kdt_WT}.
\end{proof}
\Cref{thm: Koop_act_wt} provides a closed-form representation of the Koopman action in terms of the wavelet transform. Consequently, we can characterize the action of Koopman semigroup on $g$ entirely from the wavelet transform of the measured output $y(t,\bb{x}_0)$. 

\subsection{Wavelet-based observables}\label{ss: wave_obs}

In this section, we will define the \emph{wavelet-based observables} $\psi_{\sigma_j}$ and motivate why they are good candidates for EDMD. Recall the closed form expression for the action of the Koopman operator proved in \Cref{thm: Koop_act_wt}. Namely, given any $\Delta t \geq 0$ and $\bb{x}\in \M$ we have
\begin{equation*}
    \mathcal{K}_{\Delta t}[g](\bb{x})
    =\frac{1}{C_{\Gamma}}\int_{\sigma\in \R} \int_{\tau\in \R} \mathcal{W}[y(\cdot,\bb{x})](\sigma,\tau+\Delta t) \Gamma_{\sigma,\tau}(0)~\d\tau \dfrac{\d \sigma}{|\sigma|}.\tag{\ref{eq: Kdt_WT}}
\end{equation*}
Observe that the action of the Koopman semigroup element $\K_{\Delta t}$ \emph{affects only the shift parameter} in the wavelet transform. Specifically, it translates the wavelet transform in time, yielding $\mathcal{W}[y(\cdot,\bb{x})](\sigma,\tau+\Delta t)$. This means that the evolution under $\K_{\Delta t}$ is fully captured by a translation in $\tau$. Thus, the action of $\K_{\Delta t}[g]$ can be recovered from \emph{how $\mathcal{W}[y(\cdot,\bb{x})](\sigma,\tau)$ varies with respect to $\tau$}.

Consider an approximation to \cref{eq: Kdt_WT} via quadrature. Note that the quadrature nodes correspond to shift and scale parameters of the wavelet coefficients. For the shift nodes $\tau_k$ we assume that $\tau_1 = 0$ and pick $\tau_{\text{min}} > 0$ such that $|\tau_k| \geq \tau_{\text{min}} > 0$ for any $k > 1$. This quadrature approximation to~\cref{eq: Kdt_WT} can be explicitly written as
\begin{equation} \label{eq: Kdtfirstapprox}
    \mathcal{K}_{\Delta t}[g](\bb{x}) \approx \sum_{j=1}^J \sum_{k=1}^N \alpha_{j,k} \mathcal{W}[y(\cdot,\bb{x})](\sigma_j,\tau_k+\Delta t) \Gamma_{\sigma_j,\tau_k}(0).
\end{equation}
We will work with $\Gamma\in L_1(\R)$ that are rapidly decaying, i.e., for every $m\in\N$ there exists a constant $C_m>0$ such that
\begin{equation*}
    |\Gamma(t)| \leq C_m(1+|t|)^{-m}.
\end{equation*}
Then for $\Gamma_{\sigma,\tau}(0)$ in \cref{eq: mothersc} we get
\begin{equation}\label{eq: mwave_bd}
    |\Gamma_{\sigma,\tau}(0)| = |\sigma|^{-1}\left|\Gamma\left(-\dfrac{\tau}{\sigma}\right)\right| \leq C_m |\sigma|^{-1}\left(1+\left|\dfrac{\tau}{\sigma}\right|\right)^{-m}.
\end{equation}
By applying the H\"{o}lder inequality for $L_1/ L_\infty$ to \cref{eq: WT}, we obtain the following bound for the wavelet coefficients: 
\begin{equation}\label{eq: bd_wc}
    |\mathcal{W}[y(\cdot,\bb{x})](\sigma,\tau)| \leq \int_{t=-\infty}^\infty |y(t,\bb{x})\overline{\Gamma_{\sigma,\tau}(t)}| \d t \leq \|y(\cdot,\bb{x})\|_{L_\infty} \|\Gamma_{\sigma,\tau}\|_{L_1} = \|y(\cdot,\bb{x})\|_{L_\infty} \|\Gamma\|_{L_1}.
\end{equation}
Note that $\|y(\cdot,\bb{x})\|_{L_\infty}$ is finite since for any $t\in \R$ and $\bb{x}\in \M$, we have
$
    y(t,\bb{x}) = g(\Phi(t,\bb{x}))
$
and $g$ is assumed to be bounded. For any $j= 1,\dots,J$ and $k = 1,\dots N$, we can bound that element of the sum as
\begin{equation*}
\begin{aligned}
    |\mathcal{W}[y(\cdot,\bb{x})](\sigma_j,\tau_k+\Delta t) \Gamma_{\sigma_j,\tau_k}(0)| \leq |\mathcal{W}[y(\cdot,\bb{x})](\sigma_j,\tau_k+\Delta t)| |\Gamma_{\sigma_j,\tau_k}(0)|.
\end{aligned}
\end{equation*}
Applying \cref{eq: bd_wc,eq: mwave_bd} we get the following estimate for any $m \in \N$:
\begin{equation*}
\begin{aligned}
    |\mathcal{W}[y(\cdot,\bb{x})](\sigma_j,\tau_k+\Delta t)| |\Gamma_{\sigma_j,\tau_k}(0)|\leq \|y(\cdot,\bb{x})\|_{L_\infty} \|\Gamma\|_{L_1}C_m |\sigma|^{-1}\left(1+\left|\dfrac{\tau_k}{\sigma_j}\right|\right)^{-m}.
\end{aligned}
\end{equation*}
Then for any $m\in \N$, we can write
\begin{equation*}
\begin{aligned}
    \Bigl|\sum_{j=1}^J \sum_{k=2}^N \alpha_{j,k} \mathcal{W}[y(\cdot,\bb{x})]&(\sigma_j,\tau_k+\Delta t) \Gamma_{\sigma_j,\tau_k}(0)\Bigr| \\
                &\leq \sum_{j=1}^J \sum_{k=2}^N |\alpha_{j,k}| \left|\mathcal{W}[y(\cdot,\bb{x})](\sigma_j,\tau_k+\Delta t)\right|\left|\Gamma_{\sigma_j,\tau_k}(0)\right| \\
                &= \sum_{j=1}^J \sum_{k=2}^N |\alpha_{j,k}| \|y(\cdot,\bb{x})\|_{L_\infty}\|\Gamma\|_{L_1}C_m|\sigma|^{-1}\left(1+\left|\dfrac{\tau_k}{\sigma_j}\right|\right)^{-m}.
\end{aligned}
\end{equation*}
Recall that we picked $\tau_{\text{min}}$ such that $|\tau_k| \geq \tau_{\text{min}}$. Using that fact, we bound this uniformly with respect to $\tau$ as
\begin{equation*}
\begin{aligned}
    \Bigl|\sum_{j=1}^J \sum_{k=2}^N \alpha_{j,k} \mathcal{W}[y(\cdot,\bb{x})]&(\sigma_j,\tau_k+\Delta t) \Gamma_{\sigma_j,\tau_k}(0)\Bigr| \\
                &\leq \sum_{j=1}^J \sum_{k=2}^N |\alpha_{j,k}| \|y(\cdot,\bb{x})\|_{L_\infty}\|\Gamma\|_{L_1}C_m|\sigma|^{-1}\left(1+\left|\dfrac{\tau_{\text{min}}}{\sigma_j}\right|\right)^{-m} \\
                &= \|y(\cdot,\bb{x})\|_{L_\infty}\|\Gamma\|_{L_1}C_m|\sigma|^{-1}\sum_{j=1}^J \left(1+\left|\dfrac{\tau_{\text{min}}}{\sigma_j}\right|\right)^{-m} \sum_{k=2}^N |\alpha_{j,k}|. 
\end{aligned}
\end{equation*}
Since $m\in \N$ was arbitrary, the contribution of terms that involve $\tau_k$ when $k > 1$ is negligible. Hence, one can disregard the terms with $\tau_k$ for $k > 1$ in the inner sum at \cref{eq: Kdtfirstapprox} and obtain the approximation 
\begin{equation}\label{eq: instant_freq}
    \mathcal{K}_{\Delta t}[g](\bb{x}) \approx \sum_{j=1}^J \alpha_{j,1}\mathcal{W}[y(\cdot,\bb{x})](\sigma_j,\Delta t) \Gamma_{\sigma_j,0}(0).
\end{equation}

In summary, we have shown that under some assumptions on $\Gamma$, equation \cref{eq: Kdt_WT} can be closely approximated by \cref{eq: instant_freq} for any $\Delta t \geq 0$. This amounts to saying that $\mathcal{K}_{\Delta t}[g](\bb{x})$ is approximately in the span of $\mathcal{W}[y(\cdot,\bb{x})](\sigma_j,\Delta t)$ where $\sigma_j$ are the quadrature nodes for the outer integral approximation in \cref{eq: Kdtfirstapprox}. Let $\psi_{\sigma_j}$ be defined as
\begin{equation*}
    \psi_{\sigma_j}(\bb{x}) =
    \mathcal{W}[y(\cdot,\bb{x})](\sigma_j,0),~ \textrm{for } j = 1,\dots, M.
\end{equation*}
In other words, $\psi_{\sigma_j}(\bb{x})$ is the wavelet transform applied to the output function with initial condition $\bb{x}$ at scale parameter $\sigma_j$ and shift parameter $\tau = 0$. Since $\mathcal{W}[y(\cdot,\bb{x})](\sigma_j,\Delta t)$ are the $\Delta t$ time-shifted version of $\mathcal{W}[y(\cdot,\bb{x})](\sigma_j,0)$, we expect to have 
\begin{equation}\label{eq: Kdt_wave}
      \K_{\Delta t}[\psi_{\sigma_j}] = \mathcal{W}[y(\cdot,\bb{x})](\sigma_j,\Delta t).
\end{equation}
This will be made precise later in \cref{lem: Koop_wt_obs}. 

Hence \cref{eq: instant_freq} amounts to saying that $\K_{\Delta t}[g]$ is approximately in the span of $\K_{\Delta t}[\psi_{\sigma_j}]$. So, any Koopman invariant subspace that contains $\psi_{\sigma_j}$ approximately contains $\K_{\Delta t}[g]$ for all $\Delta t > 0$. This suggests that the subspace captures the output behavior of \cref{eq: out_DS}. Thus the subspace is a good candidate for EDMD approximation. Then, if we pick our observables as a finite subset of the \emph{wavelet-based observables} defined for any $\sigma\in \R$ as  
\begin{equation}\label{eq: wave_obs}
    \psi_{\sigma}(\bb{x}) =
    \mathcal{W}[y(\cdot,\bb{x})](\sigma,0),
\end{equation}
we should expect to have a subspace that describes \cref{eq: out_DS} well, provided we choose $\sigma_j$ accordingly. Because of this, we investigate the properties of these wavelet-based observables \cref{eq: wave_obs}.

The following lemma proves the previous assertion \cref{eq: Kdt_wave}. Moreover, it gives us a way to write the action of the Koopman semigroup on $\psi_{\sigma}$ in a more tangible form.
\begin{lemma}[Koopman semigroup on observables $\psi_{\sigma}$]\label{lem: Koop_wt_obs}
    Let $\psi_{\sigma}$ be defined as in \cref{eq: wave_obs} and consider the dynamical system \cref{eq: out_DS} with $\T$ being globally Lipschitz and $g$ continuous and let $\Gamma \in L_1(\R)\cap L_2(\R)$. Then, for any $\Delta t > 0$ the action of the associated Koopman semigroup element $\K_{\Delta t}$ can be written as
    \begin{equation*}
        \K_{\Delta t}[\psi_\sigma](\bb{x}) = \mathcal{W}[y(\cdot,\bb{x})](\sigma,\Delta t). \tag{\ref{eq: Kdt_wave}}
    \end{equation*}
    Moreover, the action of the Koopman generator on $\psi_{\sigma}$ at any point $\bb{x}\in\M$ can be expressed as
    \begin{equation}\label{eq: Koop_obs}
        \mathcal{K}[\psi_{\sigma}](\bb{x}) = \lim_{\Delta t\downarrow 0} \int_{t\in\R} y(t,\bb{x}) \overline{\left(\dfrac{\Gamma_{\sigma,0}(t-\Delta t) - \Gamma_{\sigma,0}(t)}{\Delta t}\right)}\d t.
    \end{equation}
    where $\Gamma_{\sigma,0}$ is as defined in~\cref{eq: mothersc}. 
\end{lemma}
\begin{proof}
    Using~\cref{eq: Kdelta}, we can write the action of any element of the semigroup $\mathcal{K}_{\Delta t}$ on $\psi_{\sigma}$ as
    \begin{equation*} 
        \mathcal{K}_{\Delta t}[\psi_{\sigma}](\bb{x}) = \int_{t\in\R} y(t,\bb{x}_+)\overline{\Gamma_{\sigma,0}(t)} \d t = \int_{t\in\R} y(t+\Delta t,\bb{x})\overline{\Gamma_{\sigma,0}(t)} \d t,
    \end{equation*}
    where $\bb{x}_+ := \Phi(\Delta t,\bb{x})$. Applying the change of variables $t \to t-\Delta t$ we can rewrite this integral as
    \begin{equation}\label{eq: Kdeltapsi}
        \int_{t\in\R} y(t+\Delta t,\bb{x})\overline{\Gamma_{\sigma,0}(t)} \d t = \int_{t\in\R} y(t,\bb{x})\overline{\Gamma_{\sigma,0}(t-\Delta t)} \d t.
    \end{equation}
    Using \cref{eq: mothersc} and \cref{eq: WT} we can rewrite \cref{eq: Kdeltapsi} as 
    \begin{equation*}
        \int_{t\in\R} y(t,\bb{x})\overline{\Gamma_{\sigma,0}(t-\Delta t)} \d t = \int_{t\in\R} y(t,\bb{x})\overline{\Gamma_{\sigma,\Delta t}(t)} \d t = \mathcal{W}[y(\cdot,\bb{x})](\sigma,\Delta t).
    \end{equation*}
    Hence \cref{eq: Kdt_wave} holds. 
    The generator is defined as in~\cref{eq:Kcont}. Applying~\cref{eq: Kdeltapsi} in \cref{eq:Kcont} we obtain
    \begin{equation*}
        \mathcal{K}[\psi_{\sigma}](\bb{x}) = \lim_{\Delta t\downarrow 0} \dfrac{\mathcal{K}_{\Delta t}[\psi_{\sigma}](\bb{x}) - \psi_{\sigma}(\bb{x})}{\Delta t}= \lim_{\Delta t\downarrow 0} \int_{t\in\R} y(t,\bb{x}) \overline{\left(\dfrac{\Gamma_{\sigma,0}(t-\Delta t) - \Gamma_{\sigma,0}(t)}{\Delta t}\right)} \d t
    \end{equation*}  
    proving \cref{eq: Koop_obs}.
\end{proof}
\cref{lem: Koop_wt_obs} tells us that the time shifted wavelet transform describes the evolution of the wavelet-based observables $\psi_\sigma$ with respect to the Koopman semigroup hence giving us a simple expression for the action of the Koopman semigroup over $\psi_{\sigma}$. For the case where $y(\cdot,\bb{x})$ and $\Gamma_{\sigma,0}$ are smooth enough, we can switch the order of limits to obtain
\begin{equation*}
    \mathcal{K}[\psi_{\sigma}](\bb{x}) = -\int_{t\in \R} y(t,\bb{x}) \overline{\left( \dfrac{\d}{\d t}\Gamma_{\sigma,0}(t)\right)} \d t.
\end{equation*}
Using \cref{eq: WT}, this can be equivalently written as the derivative of the wavelet transform at zero with respect to the shift parameter, i.e.,
\begin{equation}\label{eq: cKoop_smooth}
    \mathcal{K}[\psi_{\sigma}](\bb{x}) = \dfrac{\d}{\d \tau}\mathcal{W}[y(\cdot,\bb{x})](\sigma,0).
\end{equation}
In wavelet analysis, the shift derivative of the wavelet is used in singularity detection \cite{SingDetect}. Moreover, it is used in \cite{wsst} for post-processing the wavelet transform leading to the wavelet synchrosqueezing transform (WSST). Simply put, in WSST one reassigns the shift and scale parameters after the wavelet transform to ``sharpen'' the time-frequency representation. For this reallocation, they use the shift derivative of the wavelet. Indeed, they prove that under some assumptions on the signal $f$, one can estimate the instantaneous frequency $\omega_f(\sigma,\tau)$ as
\begin{equation*}
    \omega_f(\sigma,\tau) \approx -\imunit \left(\mathcal{W}[f](\sigma,\tau)\right)^{-1}\dfrac{\d}{\d \tau}\mathcal{W}[f](\sigma,\tau).
\end{equation*}
when $\mathcal{W}[f](\sigma,\tau)\neq 0$ \cite[Theorem 3.3]{wsst}. Evaluating this at $\tau = 0$ we obtain
\begin{equation*}
    \dfrac{\d}{\d \tau}\mathcal{W}[f](\sigma,0) \approx \imunit \omega_f(\sigma,0)\mathcal{W}[f](\sigma,0).
\end{equation*}
For a given $\bb{x}$, if we let $f(\cdot) = y(\cdot,\bb{x})$ and apply \cref{eq: cKoop_smooth} this becomes 
\begin{equation}\label{eq: sstKoop}
    \mathcal{K}[\psi_{\sigma}](\bb{x}) = \dfrac{\d}{\d \tau}\mathcal{W}[y(\cdot,\bb{x})](\sigma,0) \approx \imunit \omega_f(\sigma,0)\mathcal{W}[y(\cdot,\bb{x})](\sigma,0) = \imunit \omega_f(\sigma,0)\psi_{\sigma}(\bb{x}).
\end{equation}
Hence the Koopman operator acts approximately like a multiplication operator on these wavelet-based observables. This special structure will help us understand the spectral properties of the Koopman operator. 

\subsection{Wavelet-based observables for modulated Gaussians}\label{ss: wbobs_Morl}
In this section, we start by analyzing \cref{eq: sstKoop} for a specific $\Gamma$ using the $C_0$-semigroup structure proposed by \cref{thm: KoopC0ode}. By this analysis, we show that for this specific $\Gamma$, $\psi_\sigma$ becomes the eigenfunctions of the associated Koopman semigroup $(\K_{\Delta t})_{\Delta t \geq 0}$ on $(C(\Omega),\|\cdot\|_\infty)$.

For a $\omega_0\in \R$ we will denote by $\Psi(\cdot; \omega_0)$ the modulated Gaussian, i.e.,
\begin{equation}\label{eq: modGauss}
    \Psi(t;\omega_0) := e^{\imunit\omega_0 t}e^{-t^2/2}
\end{equation}
with $\imunit$ denoting the imaginary unit. We will start by showing that the wavelet-based observables $\psi_\sigma$ are continuous if we let $\Gamma = \Psi(~\cdot~;\omega_0)$ for any $\omega_0\in \R$ and hence $\psi_\sigma\in C(\Omega)$ for any $\Omega\subseteq \M$ and $\sigma\neq 0$.
\begin{lemma}[Lipschitz continuity of wavelet-based observables for modulated Gaussian]\label{lem: wave_obs_cont}
    For any $\omega_0\in \mathbb{R}$, let $\Gamma=\Psi(~\cdot~;\omega_0)$ be the modulated Gaussian. Consider the dynamical system \cref{eq: out_DS}, and assume that $\mathcal{T}$ and $g$ are globally Lipschitz with Lipschitz constants $L$ and $M$, respectively. Then, for every $\sigma\neq 0$, the wavelet-based observables are Lipschitz continuous on $\mathcal{M}$.
\end{lemma}
\begin{proof}
    Since $\T$ is Lipschitz continuous, the ODE at \cref{eq: out_DS} depends continuously on the initial condition. Specifically, for two initial conditions $\bb{x},\bb{z}\in \M$, and $t \in \R$ we have the following bound
    \begin{equation*}
        \|\Phi(t,\bb{x}) - \Phi(t,\bb{z})\| \leq \exp(L|t|)\|\bb{x}-\bb{z}\|,
    \end{equation*}
    where $L$ denotes the Lipschitz constant of $\T$. Proof of this can be found in any standard ODE textbook, see, e.g., \cite[Theorem 2.8]{TeschlODE}. Then, since $g$ is also Lipschitz continuous with Lipschitz constant $M$ we have that
    \begin{equation}\label{eq: odeCIC}
    \begin{aligned}
        |y(t,\bb{x}) - y(t,\bb{z})| &= |g\circ\Phi(t,\bb{x}) - g\circ\Phi(t,\bb{x})| \\
                    &\leq M\|\Phi(t,\bb{x}) - \Phi(t,\bb{z})\|  \\
                    &\leq M\exp(L|t|)\|\bb{x}-\bb{z}\|.
    \end{aligned}
    \end{equation}
    Then, for a fixed $\sigma > 0$, consider $|\psi_\sigma(\bb{x}) - \psi_\sigma(\bb{z})|$. Using \cref{eq: wave_obs}, we can rewrite this as
    \begin{align*}
        |\psi_\sigma(\bb{x}) - \psi_\sigma(\bb{z})| = |\sigma|^{-1}\left|\int_{t=-\infty}^\infty \left(y(t,\bb{x})- y(t,\bb{z})\right)\exp(-\imunit \omega_0 t/\sigma)\exp( - t^2/2\sigma^2) \d t\right|.
    \end{align*}
    We can push the absolute value inside the integral to get
    \begin{equation*}
        |\psi_\sigma(\bb{x}) - \psi_\sigma(\bb{z})| \leq |\sigma|^{-1}\int_{t=-\infty}^\infty |y(t,\bb{x})- y(t,\bb{z})||\exp( - t^2/2\sigma^2)| \d t.
    \end{equation*}
    Using \cref{eq: odeCIC} we can bound the integral by
    \begin{align*}
        |\psi_\sigma(\bb{x}) - \psi_\sigma(\bb{z})| &\leq |\sigma|^{-1}\int_{t=-\infty}^\infty M\exp(L|t|)\|\bb{x}-\bb{z}\||\exp( - t^2/2\sigma^2)| \d t \\
                        &= |\sigma|^{-1}M\|\bb{x}-\bb{z}\|\int_{t=-\infty}^\infty \exp(L|t| - t^2/2\sigma^2) \d t.
    \end{align*}
    Since the integral at the right is convergent, $\psi_\sigma$ is Lipschitz continuous on $\M$.
\end{proof}
Note that this proof can be extended to the case where $g$ is continuous if we relax \cref{lem: wave_obs_cont} and let $\psi_\sigma$ to be continuous rather than Lipschitz continuous. 

\Cref{lem: wave_obs_cont} together with \cref{eq: bd_wc} suggests that, for this choice of $\Gamma$, the wavelet-based observables are bounded, continuous functions for any $\M$. Consequently, these observables lie in many of the function spaces where Koopman analysis is typically performed. Our next step is to show that $\psi_\sigma$ is invariant under the action of the Koopman semigroup. This will then prove that observables $\psi_\sigma$ are good candidates for the EDMD approximation. In order to show that, we first derive the following bound.

\begin{theorem}\label{thm: KoopEfn_point}
Consider the dynamical system \cref{eq: out_DS} and assume that $\T$ is globally Lipschitz and $g$ is continuous. Let $\Gamma = \Psi(~\cdot~;\omega_0)$ be the modulated Gaussian for a fixed $\omega_0\in \R$. Then the wavelet-based observables $\psi_{\sigma}(\bb{x})$ as defined in \cref{eq: wave_obs} satisfy the following bound for any $\bb{x}\in \M$ and $\Delta t>0$
\begin{equation}\label{eq: wt_inv_point}
    \left|\K_{\Delta t}[\psi_{\sigma}](\bb{x}) - \exp\left(\dfrac{\imunit w_0\Delta t}{\sigma}-\dfrac{\Delta t^2}{2\sigma^2}\right)\psi_{\sigma}(\bb{x})\right| \leq |\sigma|\sqrt{2\pi}\left(\exp\left(\dfrac{\Delta t^2}{2\sigma^2}\right) - 1\right)\max_{t \in\R} |y(t,\bb{x})|.
\end{equation}
Moreover this bound is quadratic in $\Delta t$.
\end{theorem}
\begin{proof}
    By \cref{lem: Koop_wt_obs} we have
    \begin{equation*}
    \begin{aligned}
        \K_{\Delta t}[\psi_\sigma](\bb{x}) &= \mathcal{W}[y(\cdot,\bb{x})](\sigma,\Delta t) \\
        &= |\sigma|^{-1}\int_{t\in\R}y(t,\bb{x})\exp\left(\dfrac{-\imunit\omega_0(t-\Delta t)}{\sigma}\right)\exp\left(-\dfrac{(t-\Delta t)^2}{2\sigma^2}\right) \d t.
    \end{aligned}
    \end{equation*}
    By distributing the square and taking the constants out of the integral we get
    \begin{align}
        \K_{\Delta t}&[\psi_\sigma](\bb{x}) = |\sigma|^{-1}\exp\left(\dfrac{\imunit\omega_0 \Delta t}{\sigma}\right)\int_{t\in\R}y(t,\bb{x})\exp\left(\dfrac{-\imunit\omega_0t}{\sigma}\right)\exp\left(-\dfrac{(t-\Delta t)^2}{2\sigma^2}\right) \d t \nonumber\\
                &= |\sigma|^{-1}\exp\left(\dfrac{\imunit\omega_0 \Delta t}{\sigma}\right)\int_{t\in\R}y(t,\bb{x})\exp\left(\dfrac{-\imunit\omega_0t}{\sigma}\right)\exp\left(-\dfrac{t^2-2t\Delta t+\Delta t^2}{2\sigma^2}\right) \d t \nonumber\\
                &= |\sigma|^{-1}\exp\left(\dfrac{\imunit\omega_0 \Delta t}{\sigma}-\dfrac{\Delta t^2}{2\sigma^2}\right)\int_{t\in\R}y(t,\bb{x})\exp\left(\dfrac{-\imunit\omega_0t}{\sigma}\right)\exp\left(-\dfrac{t^2}{2\sigma^2}\right)\exp\left(\dfrac{t \Delta t}{\sigma^2}\right) \d t.\label{eq: Kdt_wobs_first}
    \end{align}
    We focus on the integral in \cref{eq: Kdt_wobs_first}. By adding and subtracting one we get
    \begin{equation*}
    \begin{aligned}
        \int_{t\in\R}y(t,\bb{x})&\exp\left(\dfrac{-\imunit\omega_0t}{\sigma}\right)\exp\left(-\dfrac{t^2}{2\sigma^2}\right)\exp\left(\dfrac{t \Delta t}{\sigma^2}\right) \d t \\ 
            &= \int_{t\in\R}y(t,\bb{x})\exp\left(\dfrac{-\imunit\omega_0t}{\sigma}\right)\exp\left(-\dfrac{t^2}{2\sigma^2}\right)\left(\exp\left(\dfrac{t \Delta t}{\sigma^2}\right)-1+1\right) \d t \\
            &= \int_{t\in\R}y(t,\bb{x})\exp\left(\dfrac{-\imunit\omega_0t}{\sigma}\right)\exp\left(-\dfrac{t^2}{2\sigma^2}\right)\left(\exp\left(\dfrac{t \Delta t}{\sigma^2}\right)-1\right) \d t \\
            &\quad+ \int_{t\in\R}y(t,\bb{x})\exp\left(\dfrac{-\imunit\omega_0t}{\sigma}\right)\exp\left(-\dfrac{t^2}{2\sigma^2}\right) \d t.
    \end{aligned}
    \end{equation*}
    The second integral is exactly the wavelet-based observable evaluated at $\bb{x}$ scaled by $|\sigma|$, i.e., $|\sigma|\psi_\sigma(\bb{x})$. Hence we can rewrite \cref{eq: Kdt_wobs_first} as
    \begin{align*}
        \K_{\Delta t}[\psi_\sigma](\bb{x}) - \exp&\left(\dfrac{\imunit\omega_0 \Delta t}{\sigma}-\dfrac{\Delta t^2}{2\sigma^2}\right)\psi_\sigma(\bb{x}) =|\sigma|^{-1}\exp\left(\dfrac{\imunit\omega_0 \Delta t}{\sigma}-\dfrac{\Delta t^2}{2\sigma^2}\right)\\
            &\times\int_{t\in\R}y(t,\bb{x})\exp\left(\dfrac{-\imunit\omega_0t}{\sigma}\right)\exp\left(-\dfrac{t^2}{2\sigma^2}\right)\left(\exp\left(\dfrac{t \Delta t}{\sigma^2}\right)-1\right) \d t.
    \end{align*}
    Then, by taking the absolute values of both sides, we have the following bound
    \begin{equation}\label{eq: eFn_errbd}
    \begin{aligned}
        \Bigl| \K_{\Delta t}[\psi_\sigma](\bb{x}) &- \exp\left(\dfrac{\imunit\omega_0 \Delta t}{\sigma}-\dfrac{\Delta t^2}{2\sigma^2}\right)\psi_\sigma(\bb{x})\Bigr| \leq |\sigma|^{-1}\left|\exp\left(\dfrac{\imunit\omega_0 \Delta t}{\sigma}-\dfrac{\Delta t^2}{2\sigma^2}\right)\right| \\
            &\times\left|\int_{t\in\R}y(t,\bb{x})\exp\left(\dfrac{-\imunit\omega_0t}{\sigma}\right)\exp\left(-\dfrac{t^2}{2\sigma^2}\right)\left(\exp\left(\dfrac{t \Delta t}{\sigma^2}\right)-1\right) \d t\right|.
    \end{aligned}
    \end{equation}
    Since $y$ is bounded, one can use the following bound to approximate the integral in \cref{eq: eFn_errbd}:
    \begin{align}
        \Bigl|\int_{t\in\R}y(t,\bb{x})&\exp\left(\dfrac{-\imunit\omega_0t}{\sigma}\right)\exp\left(-\dfrac{t^2}{2\sigma^2}\right)\left(\exp\left(\dfrac{t \Delta t}{\sigma^2}\right)-1\right) \d t\Bigr| \nonumber\\
            &\leq \left(\max_{t \geq 0} \Bigl|y(t,\bb{x})\exp\left(\dfrac{-\imunit\omega_0t}{\sigma}\right)\Bigr|\right) \Bigl|\int_{t\in\R}\exp\left(-\dfrac{t^2}{2\sigma^2}\right)\left(\exp\left(\dfrac{t \Delta t}{\sigma^2}\right)-1\right) \d t\Bigr|\nonumber \\
            &\leq \left(\max_{t \geq 0} |y(t,\bb{x})|\right) \Bigl|\int_{t\in\R}\exp\left(-\dfrac{t^2}{2\sigma^2}\right)\left(\exp\left(\dfrac{t \Delta t}{\sigma^2}\right)-1\right) \d t\Bigr|.\label{eq: inv_bound_1}
    \end{align}
    We can split the integral in \cref{eq: inv_bound_1} and complete the square inside the exponential to obtain
    \begin{align}
        \int_{t\in\R}\exp\left(-\dfrac{t^2}{2\sigma^2}\right)&\left(\exp\left(\dfrac{t \Delta t}{\sigma^2}\right)-1\right) \d t \nonumber \\ 
        &= \int_{t\in\R}\exp\left(-\dfrac{t^2}{2\sigma^2}\right)\exp\left(\dfrac{t \Delta t}{\sigma^2}\right) \d t - \int_{t\in\R}\exp\left(-\dfrac{t^2}{2\sigma^2}\right) \d t.\label{eq: split_exp_1}
    \end{align}
    The second integral is $\sqrt{2\pi}|\sigma|$. For the first one, we complete the square.
    \begin{align*}
        \int_{t\in\R}\exp\left(-\dfrac{t^2}{2\sigma^2}\right)\exp\left(\dfrac{t \Delta t}{\sigma^2}\right) \d t &= \int_{t\in\R}\exp\left(\dfrac{t \Delta t}{\sigma^2}-\dfrac{t^2}{2\sigma^2}\right) \d t \\
            &= \int_{t\in\R}\exp\left(\dfrac{\Delta t^2}{2\sigma^2}-\dfrac{(t-\Delta t)^2}{2\sigma^2}\right) \d t \\
            &= \exp\left(\dfrac{\Delta t^2}{2\sigma^2}\right)\int_{t\in\R}\exp\left(-\dfrac{(t-\Delta t)^2}{2\sigma^2}\right) \d t.
    \end{align*}
    Again, the integral is known to be $\sqrt{2\pi}|\sigma|$. Hence \cref{eq: split_exp_1} becomes
    \begin{equation*}
        \int_{t\in\R}\exp\left(-\dfrac{t^2}{2\sigma^2}\right)\left(\exp\left(\dfrac{t \Delta t}{\sigma^2}\right)-1\right) \d t = \sqrt{2\pi}|\sigma|\left(\exp\left(\dfrac{\Delta t^2}{2\sigma^2}\right) - 1\right).
    \end{equation*}
    Thus the expression in \cref{eq: inv_bound_1} simplifies to 
    \begin{equation}\label{eq: inv_bound_2}
        \sqrt{2\pi}\left(\exp\left(\dfrac{\Delta t^2}{2\sigma^2}\right) - 1\right)\max_{t \in\R} |y(t,\bb{x})|.
    \end{equation}
    Note that the exponential term in \cref{eq: eFn_errbd} can be calculated as
    \begin{equation*}
        \left|\exp\left(\dfrac{\imunit\omega_0 \Delta t}{\sigma}-\dfrac{\Delta t^2}{2\sigma^2}\right)\right| = \exp\left(\Re\left(\dfrac{\imunit\omega_0 \Delta t}{\sigma}-\dfrac{\Delta t^2}{2\sigma^2}\right)\right) =\exp\left(-\dfrac{\Delta t^2}{2\sigma^2}\right).
    \end{equation*}
    Also note that this is a negative value inside exponential. Hence
    \begin{equation*}
        \left|\exp\left(\dfrac{\imunit\omega_0 \Delta t}{\sigma}-\dfrac{\Delta t^2}{2\sigma^2}\right)\right| = \exp\left(-\dfrac{\Delta t^2}{2\sigma^2}\right) \leq 1.
    \end{equation*}
    This together with \cref{eq: inv_bound_2} suggests that \cref{eq: eFn_errbd} simplifies to
    \begin{equation}\label{eq: pw_err_bd}
    \begin{aligned}
        \Bigl|\K_{\Delta t}[\psi_\sigma](\bb{x}) - \exp&\left(\dfrac{\imunit\omega_0 \Delta t}{\sigma}-\dfrac{\Delta t^2}{2\sigma^2}\right)\psi_\sigma(\bb{x})\Bigr| \\
        &\leq |\sigma|\sqrt{2\pi}\left(\exp\left(\dfrac{\Delta t^2}{2\sigma^2}\right) - 1\right)\max_{t \in\R} |y(t,\bb{x})|.
    \end{aligned}
    \end{equation}
    To show that the error decays quadratically with respect to $\Delta t$, consider the Taylor expansion of the exponential part:
    \begin{equation*}
        \exp\left(\dfrac{\Delta t^2}{2\sigma^2}\right) - 1 = \left(\sum_{n=0}^\infty \dfrac{\Delta t^{2n}}{n!2^n\sigma^{2n}}\right) -1
    \end{equation*}
    Since for $n=0$ the summand is $1$, we can rewrite this as
    \begin{equation*}
        \exp\left(\dfrac{\Delta t^2}{2\sigma^2}\right) - 1  = \sum_{n=1}^\infty \dfrac{\Delta t^{2n}}{n!2^n\sigma^{2n}}.
    \end{equation*}
    Hence the exponential term is on the order of $\Delta t^2$. Since that is the only $\Delta  t$ dependent term in \cref{eq: wt_inv_point}, we obtain quadratic convergence as $\Delta t \to 0$.
\end{proof}

This result shows that the residual at \cref{eq: wt_inv_point} is decaying quadratically \emph{for each point $\bb{x}\in \M$} as $\Delta t \to 0$. Because this decay is faster than linear, it suggests that these observables are natural candidates for eigenfunctions of the Koopman generator, and hence of the semigroup. We emphasize, however, that \cref{thm: KoopEfn_point} is a pointwise statement. By contrast, to establish that a function is an eigenfunction of the $C_0$-semigroup $(\K_{\Delta t})_{\Delta t \geq 0}$ over $C(\Omega)$, we need to verify the convergence in the uniform norm $\|\cdot\|_\infty$ over $\Omega$. For the case where $g$ is bounded, the $\|\cdot\|_\infty$ norm convergence follows from \cref{thm: KoopEfn_point}. The next theorem indeed shows that $\psi_\sigma$ are eigenfunctions of the semigroup $(\K_{\Delta t})_{\Delta t\geq 0}$ on $C(\Omega)$ where $\Omega\subseteq \M$ is compact and forward-invariant.

\begin{theorem}\label{thm: KoopEfn}
    Consider the dynamical system \cref{eq: out_DS} and assume that $\T$ is globally Lipschitz and $g$ is continuous. For a fixed $\omega_0\in \R$, let $\Gamma = \Psi(~\cdot~;\omega_0)$ be the modulated Gaussian and $\Omega$ be a compact forward-invariant subset of $\M$. Then for any $\sigma\neq 0$ the wavelet-based observable $\psi_{\sigma}$ as defined in \cref{eq: wave_obs} are eigenfunctions of the Koopman semigroup over $C(\Omega)$ with eigenvalue $\dfrac{\imunit \omega_0}{\sigma}$, i.e., for the generator $\K$ we have $\psi_\sigma\in \mathcal{D}(\K)$ and
    \begin{equation}\label{eq: KoopEfn}
        \K[\psi_\sigma] = \dfrac{\imunit \omega_0}{\sigma}\psi_{\sigma}.
    \end{equation}
\end{theorem}
\begin{proof}
    Since \cref{eq: out_DS} satisfy the assumptions of \cref{thm: KoopC0ode} we have that $(\K_{\Delta t})_{t\geq 0}$ is a $C_0$-semigroup over $(C(\Omega),\|\cdot\|_\infty)$. Our goal is to derive the $\|\cdot\|_\infty$ bound from the pointwise bound at \cref{thm: KoopEfn_point}. To do this, take the supremum norm for both sides of \cref{eq: pw_err_bd} on $\Omega\subseteq \M$ to obtain
    \begin{align}
        \Bigl\|\K_{\Delta t}[\psi_\sigma] - \exp\left(\dfrac{\imunit\omega_0 \Delta t}{\sigma}-\dfrac{\Delta t^2}{2\sigma^2}\right)\psi_\sigma\Bigr\|_\infty    &= \sup_{\bb{x}\in \Omega}\Bigl|\K_{\Delta t}[\psi_\sigma](\bb{x}) - \exp\left(\dfrac{\imunit\omega_0 \Delta t}{\sigma}-\dfrac{\Delta t^2}{2\sigma^2}\right)\psi_\sigma(\bb{x})\Bigr|\nonumber \\
                    &\leq \sup_{\bb{x}\in \Omega}|\sigma|\sqrt{2\pi}\left(\exp\left(\dfrac{\Delta t^2}{2\sigma^2}\right) - 1\right)\max_{t \in\R} |y(t,\bb{x})| \nonumber.
    \end{align}
    Since only $\bb{x}$ dependent term on the bound is $y(t,\bb{x})$, we can take constant terms out:
    \begin{equation}\label{eq: sup_err_bd}
        \Bigl\|\K_{\Delta t}[\psi_\sigma] - \exp\left(\dfrac{\imunit\omega_0 \Delta t}{\sigma}-\dfrac{\Delta t^2}{2\sigma^2}\right)\psi_\sigma\Bigr\|_\infty \leq |\sigma|\sqrt{2\pi} \left(\exp\left(\dfrac{\Delta t^2}{2\sigma^2}\right) - 1\right)\sup_{\bb{x}\in \Omega}\max_{t \in\R} |y(t,\bb{x})|.
    \end{equation}
    Since $y(t,\bb{x}) = g(\Phi(t,\bb{x})) = g(\Phi(0,\bb{x}_0)) = g(\bb{x}_0)$ for some $\bb{x}_0\in \R^n$, the $\bb{x}$ dependent term in \cref{eq: sup_err_bd} is bounded by
    \begin{equation*}
        \sup_{\bb{x}\in \Omega}\max_{t \in\R} |y(t,\bb{x})| = \sup_{\bb{x}\in \Omega}\max_{t \in\R} |g\circ\Phi(t,\bb{x})| \leq \sup_{\bb{x}\in \M} |g(\bb{x})| \leq M <\infty,
    \end{equation*}
    for some $M\in \R_+$. Existence of such $M$ is guaranteed because $g$ is bounded on $\M$. Then, with this bound, \cref{eq: sup_err_bd} becomes
    \begin{align}
        \Bigl\|\K_{\Delta t}[\psi_\sigma] - \exp\left(\dfrac{\imunit\omega_0 \Delta t}{\sigma}-\dfrac{\Delta t^2}{2\sigma^2}\right)\psi_\sigma\Bigr\|_\infty &\leq |\sigma|\sqrt{2\pi}\left(\exp\left(\dfrac{\Delta t^2}{2\sigma^2}\right) - 1\right)\sup_{\bb{x}\in \M} |g(\bb{x})| \nonumber \\
        &\leq |\sigma|\sqrt{2\pi} M \left(\exp\left(\dfrac{\Delta t^2}{2\sigma^2}\right) - 1\right).\label{eq: sup_norm_bd}
    \end{align}
    Hence, the bound is quadratic also in $C(\Omega)$ norm with respect to $\Delta t$.
    
    Recall that the generator $\K:\mathcal{D}(\K)\subset C(\Omega)\to C(\Omega)$ of the Koopman semigroup $(\K_{\Delta t})_{\Delta t \geq 0}$ is defined for any $f\in \mathcal{D}(\K)$ by
    \begin{equation*}
        \K[f] = \lim_{\Delta t \downarrow 0} \dfrac{\K_{\Delta t}[f] -f}{\Delta t},
    \end{equation*}
    where the limit is taken in the strong sense. The domain $\mathcal{D}(\K)$ is defined to be
    \begin{equation*}
        \mathcal{D}(\K) = \left\{f\in X\mid \lim_{\Delta t \downarrow 0} \dfrac{\K_{\Delta t}[f] -f}{\Delta t} \text{ exists in } (C(\Omega),\|\cdot\|_\infty)\right\}.
    \end{equation*}
    Equivalently, this means that $f\in \dom(\K)$ if and only if there exists a $h\in C(\Omega)$ such that
    \begin{equation}\label{eq: domdef}
        \sup_{\bb{x}\in \Omega}\left| \dfrac{\K_{\Delta t}[f](\bb{x}) -f(\bb{x})}{\Delta t} - h(\bb{x})\right| \to 0\quad \mbox{as}~\Delta t \to 0.
    \end{equation}
    We claim that $\psi_\sigma$ is an eigenfunction of the Koopman semigroup, specifically $\K[\psi_\sigma] = \dfrac{\imunit \omega_0 }{\sigma}\psi_\sigma$. Indeed, consider the left hand side of  \cref{eq: sup_err_bd} for any $\Delta t > 0$. We can rewrite it as
    \begin{align*}
        \Bigl\|\K_{\Delta t}[\psi_\sigma] - \exp\Bigl(\dfrac{\imunit\omega_0 \Delta t}{\sigma}-&\dfrac{\Delta t^2}{2\sigma^2}\Bigr)\psi_\sigma\Bigr\|_\infty \\
        &=  \Bigl\|\K_{\Delta t}[\psi_\sigma] - \exp\left(\dfrac{\imunit\omega_0 \Delta t}{\sigma}-\dfrac{\Delta t^2}{2\sigma^2}\right)\psi_\sigma + \psi_\sigma - \psi_\sigma\Bigr\|_\infty\\
                &= \Bigl\|\left(\K_{\Delta t}[\psi_\sigma] - \psi_\sigma\right) - \left(\exp\left(\dfrac{\imunit\omega_0 \Delta t}{\sigma}-\dfrac{\Delta t^2}{2\sigma^2}\right)\psi_\sigma - \psi_\sigma\right)\Bigr\|_\infty.
    \end{align*}
    Then by \cref{eq: sup_norm_bd} we have
    \begin{align*}
        \Bigl\|\left(\K_{\Delta t}[\psi_\sigma] - \psi_\sigma\right) - \Bigl(\exp&\left(\dfrac{\imunit\omega_0 \Delta t}{\sigma}-\dfrac{\Delta t^2}{2\sigma^2}\right) - 1\Bigr) \psi_\sigma\Bigr\|_\infty  \\
                &\leq |\sigma|\sqrt{2\pi} M \left(\exp\left(\dfrac{\Delta t^2}{2\sigma^2}\right) - 1\right).
    \end{align*}
    Dividing both sides by $\Delta t > 0$ and taking the limit as $\Delta t\downarrow 0$ we get
    \begin{align}
        \lim_{\Delta t\downarrow 0}\Bigl\|\left(\dfrac{\K_{\Delta t}[\psi_\sigma] - \psi_\sigma}{\Delta t}\right) &- \left(\dfrac{\exp\left(\dfrac{\imunit\omega_0 \Delta t}{\sigma}-\dfrac{\Delta t^2}{2\sigma^2}\right) - 1}{\Delta t}\right) \psi_\sigma\Bigr\|_\infty \nonumber \\
                &\leq |\sigma|\sqrt{2\pi} M \lim_{\Delta t\downarrow 0} \dfrac{1}{\Delta t}\left(\exp\left(\dfrac{\Delta t^2}{2\sigma^2}\right) - 1\right).\label{eq: der_Koop_wt}
    \end{align}
    The limit at the right hand side of \cref{eq: der_Koop_wt} is exactly the derivative of $\exp\left(\dfrac{\Delta t^2}{2\sigma^2}\right)$ at $\Delta t = 0$, which is
    \begin{equation*}
        \lim_{\Delta t\downarrow 0} \dfrac{1}{\Delta t}\left(\exp\left(\dfrac{\Delta t^2}{2\sigma^2}\right) - 1\right) = \left(\dfrac{\partial}{\partial \Delta t}\exp\left(\dfrac{\Delta t^2}{2\sigma^2}\right)\right)_{\Delta t=0} = \dfrac{0}{\sigma^2}\exp\left(\dfrac{0}{2\sigma^2}\right) = 0.
    \end{equation*}
     Combining this with \cref{eq: der_Koop_wt} we get
    \begin{equation*}
        \lim_{\Delta t\downarrow 0}\Bigl\|\left(\dfrac{\K_{\Delta t}[\psi_\sigma] - \psi_\sigma}{\Delta t}\right) - \left(\dfrac{\exp\left(\dfrac{\imunit\omega_0 \Delta t}{\sigma}-\dfrac{\Delta t^2}{2\sigma^2}\right) - 1}{\Delta t}\right) \psi_\sigma\Bigr\|_\infty = 0.
    \end{equation*}
    This shows that $\psi_\sigma$ satisfies \cref{eq: domdef}, thus $\psi_\sigma \in \mathcal{D}(\K)$. Moreover we have
    \begin{align*}
        \K[\psi_\sigma] &= \left(\dfrac{\partial}{\partial\Delta t}\exp\left(\dfrac{\imunit\omega_0 \Delta t}{\sigma}-\dfrac{\Delta t^2}{2\sigma^2}\right)\right)_{\Delta t = 0}\psi_\sigma \\          
            &= \left(\left(\dfrac{\imunit\omega_0}{\sigma}-\dfrac{\Delta t}{\sigma^2}\right)\exp\left(\dfrac{\imunit\omega_0 \Delta t}{\sigma}-\dfrac{\Delta t^2}{2\sigma^2}\right)\right)_{\Delta t = 0}\psi_\sigma \\
            &= \dfrac{\imunit\omega_0}{\sigma}\psi_\sigma.
    \end{align*}
\end{proof}

Recall that \cref{eq: sstKoop} suggests that the Koopman generator acts approximately as a multiplication operator on wavelet-based observables. \cref{thm: KoopEfn} expands on this and proves that, if we fix $\Gamma$ to be a modulated Gaussian, these wavelet-based observables become \emph{eigenfunctions} of the Koopman semigroup when considered on the Banach space of continuous functions. This not only shows that the subspace they span is invariant, but also provides us a method to capture the action of this semigroup. Hence these wavelet-based observables become a useful tool when analyzing the Koopman semigroup.

\subsection{Implications for the Koopman semigroup}\label{ss: Koop_wave}

\cref{thm: KoopEfn} have shown that when $\Gamma$ is picked to be the modulated Gaussian, $\psi_\sigma$ are eigenfunctions of the Koopman semigroup, hence they span an invariant subspace of $C(\Omega)$. In this section, we will analyze the Koopman semigroup and its resolvent using these observables. First, we should note that $\Gamma$ does not satisfy \cref{eq: mothermean}. Hence it does not have an inverse wavelet transform \cref{eq: iCWT_pw0} so \cref{thm: Koop_act_wt} does not hold if we choose $\Gamma$ as a modulated Gaussian. To solve this problem, we will consider the complex Morlet wavelet instead.

The complex Morlet wavelet centered at frequency $\omega_0$ is defined to be
\begin{equation}\label{eq: morletmother}
    \Gamma_M(t) = (e^{\imunit\omega_0 t} - \kappa) e^{-t^2/2},\quad \omega_0>0.
\end{equation}
with $\imunit$ denoting the imaginary unit. Here $\kappa := e^{-\omega_0^2/2}$ ensures that $\Gamma$ satisfies \cref{eq: mothermean}. Note that we have 
\begin{equation}\label{eq: morletsplit}
    \Gamma_M(t) = \Psi(t;\omega_0) - \kappa\Psi(t;0)
\end{equation}
and since our results in \cref{ss: wbobs_Morl} are for generic $\omega_0\in \R$ one can obtain the results for the Morlet wavelet $\Gamma_M$ by combining the results for $\Psi(~\cdot~; \omega_0)$ and $\Psi(~\cdot~;0)$.

First, note that \cref{eq: morletsplit} implies that for any $\bb{x}\in\M$ the associated wavelet-based observables satisfy
\begin{equation}\label{eq: morletwavesplit}
    \psi_{\sigma,M}(\bb{x}) = \psi_{\sigma,\omega_0}(\bb{x}) -\kappa \psi_{\sigma,0}(\bb{x}).
\end{equation}
Here $\psi_{\sigma,M}$ denotes the wavelet-based observable generated via Morlet wavelet \cref{eq: morletmother} and for any $\omega_0\in \R$, $\psi_{\sigma,\omega_0}$ denotes the wavelet-based observable with modulated Gaussian $\Psi(~\cdot~;\omega_0)$. Then \cref{thm: KoopEfn} implies the following corollary.
\begin{corollary}\label{cor: KoopActMorl}
    Consider the dynamical system \cref{eq: out_DS} and assume that $\T$ is globally Lipschitz and $g$ is continuous. Let $\Omega$ be a compact forward-invariant subset of $\M$. Then for any $\sigma\neq 0$ and $\Delta t > 0$ the wavelet-based observables $\psi_{\sigma,M}$ satisfy
    \begin{equation*}
        \K_{\Delta t}[\psi_{\sigma,M}] = \exp\left(\dfrac{\imunit \omega_0 \Delta t}{\sigma}\right)\psi_{\sigma,\omega_0} - \kappa \psi_{\sigma,0}.
    \end{equation*}
\end{corollary}
\begin{proof}
To see this, we apply $\K_{\Delta t}$ to $\psi_{\sigma,M}$ and use linearity of $\K_{\Delta t}$ to get
\begin{equation*}
    \K_{\Delta t}[\psi_{\sigma,M}] = \K_{\Delta t}[\psi_{\sigma,\omega_0}]-\kappa \K_{\Delta t}[\psi_{\sigma,0}].
\end{equation*}
Then by applying \cref{thm: KoopEfn} to $\psi_{\sigma,\omega_0}$ and $\psi_{\sigma,0}$ we obtain
\begin{equation*}
\begin{aligned}
    \K_{\Delta t}[\psi_{\sigma,M}] &= \exp\left(\dfrac{\imunit \omega_0\Delta t}{\sigma}\right)\psi_{\sigma,\omega_0}-\kappa\exp\left( \dfrac{\imunit 0\Delta t}{\sigma}\right)\psi_{\sigma,0} \\
    &= \exp\left(\dfrac{\imunit \omega_0\Delta t}{\sigma}\right)\psi_{\sigma,\omega_0}-\kappa\psi_{\sigma,0}.
\end{aligned}
\end{equation*}
\end{proof}

We will start with the action of a semigroup element $\K_{\Delta t}$ on $g$. Recall that for a generic wavelet-based observable we had \cref{thm: Koop_act_wt}. When we pick $\Gamma$ to be the complex Morlet wavelet \cref{eq: morletmother} this simplifies as we show next.

\begin{theorem}[Action of the Koopman semigroup for modulated Gaussian]\label{lem: Kdt_act_Morl}
    Assume the conditions at \cref{thm: Koop_act_wt} and fix an $\omega_0\in \R$. Then the action of $\K_{\Delta t}$ on $g$ simplifies to
    \begin{equation}\label{eq: Kdt_WT_Morl}
        \mathcal{K}_{\Delta t}[g](\bb{x})
        =\dfrac{\sqrt{2\pi}}{C_{\Gamma_M}}\int_{\sigma\in \R} \psi_{\sigma,\omega_0}(\bb{x}) \Bigl(\exp(\imunit \omega_0 \Delta t/\sigma) -\kappa\Bigr)~\frac{\d \sigma}{|\sigma|},~\text{for any } \Delta t \geq 0
    \end{equation}
    with $\kappa = e^{-\omega_0^2/2}$. Here $C_{\Gamma_M}$ is the admissibility constant for the \emph{complex Morlet wavelet}.
\end{theorem}
\begin{proof}
    Let $\Gamma = \Gamma_M$ be the complex Morlet wavelet. Then for a fixed $\Delta t > 0$, \cref{thm: Koop_act_wt} suggests that we have
    \begin{equation*}
        \mathcal{K}_{\Delta t}[g](\bb{x})
        =\frac{1}{C_{\Gamma_M}}\int_{\sigma\in \R} \int_{\tau\in \R} \mathcal{W}[y(\cdot,\bb{x})](\sigma,\tau+\Delta t) \Gamma_{\sigma,\tau}(0) \d \tau \dfrac{\d \sigma}{|\sigma|}.
    \end{equation*}
    Using \cref{lem: Koop_wt_obs} we can rewrite this as
    \begin{equation*}
        \mathcal{K}_{\Delta t}[g](\bb{x}_0)
        =\frac{1}{C_{\Gamma_M}}\int_{\sigma\in \R} \int_{\tau\in \R} \K_{\Delta t+\tau}[\psi_{\sigma,M}](\bb{x}) \Gamma_{\sigma,\tau}(0) \d \tau \dfrac{\d \sigma}{|\sigma|}.
    \end{equation*}
    Since $\Gamma$ is a complex Morlet wavelet, we apply \cref{cor: KoopActMorl} to get
    \begin{equation*}
        \mathcal{K}_{\Delta t}[g](\bb{x}_0)
        =\frac{1}{C_{\Gamma_M}}\int_{\sigma\in \R} \int_{\tau\in \R} \Bigl(\exp\left(\dfrac{\imunit \omega_0 (\Delta t+\tau)}{\sigma}\right)\psi_{\sigma,\omega_0} - \kappa \psi_{\sigma,0}\Bigr)\Gamma_{\sigma,\tau}(0)\d \tau \dfrac{\d \sigma}{|\sigma|}.
    \end{equation*}
    We split this integral and take $\psi_\sigma(\bb{x})$ out of the inner integral and get
    \begin{align}
        \mathcal{K}_{\Delta t}[g](\bb{x})
        &=\frac{1}{C_{\Gamma_M}}\int_{\sigma\in \R} \psi_{\sigma,\omega_0}(\bb{x})\left(\int_{\tau\in \R} \exp(\imunit \omega_0(\Delta t + \tau)/\sigma) \Gamma_{\sigma,\tau}(0) \d \tau\right) \label{eq: cmKdtint11}\\
        &\qquad\qquad-\kappa~\psi_{\sigma,0}(\bb{x})\left(\int_{\tau\in \R}  \Gamma_{\sigma,\tau}(0) \d \tau\right) \dfrac{\d \sigma}{|\sigma|}.\label{eq: cmKdtint12}
    \end{align}
    Note that since $\Gamma$ satisfies \cref{eq: mothermean} we have
    \begin{equation*}
        \int_{\tau\in \R}  \Gamma_{\sigma,\tau}(0) \d \tau = \int_{\tau\in \R}  \Gamma(-\tau/\sigma) \d \tau = 0.
    \end{equation*}
    This suggests that \cref{eq: cmKdtint12} is zero for any $\sigma$. Hence we get
    \begin{equation}\label{eq: cmKdtint2}
        \mathcal{K}_{\Delta t}[g](\bb{x})
        =\frac{1}{C_{\Gamma_M}}\int_{\sigma\in \R} \psi_{\sigma,\omega_0}(\bb{x})\left(\int_{\tau\in \R} \exp(\imunit \omega_0(\Delta t + \tau)/\sigma) \Gamma_{\sigma,\tau}(0) \d \tau\right) \dfrac{\d \sigma}{|\sigma|}.
    \end{equation}
    By \cref{eq: morletsplit}, we can split the inner integral at \cref{eq: cmKdtint2} as
    \begin{align}
        \mathcal{K}_{\Delta t}[g](\bb{x})
        &=\frac{1}{C_{\Gamma_M}}\int_{\sigma\in \R} \psi_{\sigma,\omega_0}(\bb{x})\left(\int_{\tau\in \R} \exp(\imunit \omega_0(\Delta t + \tau)/\sigma) \Psi_{\sigma,\tau}(0;\omega_0) \d \tau\right) \label{eq: cmKdtint31}\\
        &\qquad\qquad -\kappa \left(\int_{\tau\in \R} \exp(\imunit \omega_0(\Delta t + \tau)/\sigma) \Psi_{\sigma,\tau}(0;0) \d \tau\right) \dfrac{\d \sigma}{|\sigma|},\label{eq: cmKdtint32}
    \end{align}
    where $\Psi(~\cdot~;\omega_0)$ is defined as in \cref{eq: modGauss}. Consider the inner integral at \cref{eq: cmKdtint31}. Using \cref{eq: modGauss} we can calculate this as
    \begin{equation*}
    \begin{aligned}
        \int_{\tau\in \R} \exp&(\imunit \omega_0(\Delta t + \tau)/\sigma) \Psi_{\sigma,\tau}(0;\omega_0) \d \tau = \int_{\tau\in \R} \exp(\imunit \omega_0(\Delta t + \tau)/\sigma) \Psi_{\sigma,0}(-\tau;\omega_0) \d \tau \\
        &= |\sigma|^{-1}\int_{\tau\in \R} \exp(\imunit \omega_0(\Delta t + \tau)/\sigma) \exp(\imunit \omega_0 (-\tau)/\sigma)\exp(-(-\tau)^2/2\sigma^2) \d \tau.
    \end{aligned}
    \end{equation*}
    This can be simplified to
    \begin{equation*}
    \begin{aligned}
        \int_{\tau\in \R} \exp(\imunit \omega_0(\Delta t + \tau)/\sigma) \Gamma_{\sigma,\tau}(0) \d \tau &= |\sigma|^{-1}\int_{\tau\in \R} \exp(\imunit \omega_0(\Delta t + \tau-\tau)/\sigma)\exp(-\tau^2/2\sigma^2) \d \tau \\
        &= |\sigma|^{-1}\int_{\tau\in \R} \exp(\imunit \omega_0\Delta t/\sigma)\exp(-\tau^2/2\sigma^2) \d \tau \\
        &= |\sigma|^{-1}\exp(\imunit \omega_0\Delta t/\sigma)\int_{\tau\in \R} \exp(-\tau^2/2\sigma^2) \d \tau.
    \end{aligned}
    \end{equation*}
    The solution of the last integral is
    \begin{equation*}
        \int_{\tau\in \R} \exp(-\tau^2/2\sigma^2) \d \tau = |\sigma|\sqrt{2\pi}.
    \end{equation*}
    Similarly, we can calculate the inner integral at \cref{eq: cmKdtint32} using the same arguments with $\omega_0 = 0$. Then \cref{eq: cmKdtint2} becomes
    \begin{equation*}
        \mathcal{K}_{\Delta t}[g](\bb{x}) =  \dfrac{\sqrt{2\pi}}{C_{\Gamma_M}}\int_{\sigma\in \R} \psi_{\sigma,\omega_0}(\bb{x})\Bigl(\exp(\imunit \omega_0 \Delta t/\sigma)-\kappa\Bigr) \dfrac{\d \sigma}{|\sigma|}.
    \end{equation*}
    This then proves \cref{eq: Kdt_WT_Morl}.
\end{proof}
\cref{lem: Kdt_act_Morl} suggests that with the wavelet-based observables evaluated at $\bb{x}$, we can fully recover the Koopman semigroup. Hence any EDMD approximation with the wavelet-based observables will essentially act as a quadrature approximation to the integral at \cref{eq: Kdt_WT_Morl}.

We can also use \cref{lem: Kdt_act_Morl} to get a closed form expression of the resolvent in terms of the wavelet-based observables $\psi_\sigma$. 
\begin{lemma}[Characterizing the resolvent via wavelet-based observables $\psi_\sigma$]\label{thm: Koop_res_Morl}
    Assume the conditions of \cref{thm: Koop_act_wt}. Then for a given $s\in\C$ with $\mathrm{Re}(s)>0$, we can characterize the Koopman resolvent on $s$ whenever it exists as
    \begin{equation}\label{eq: Kres_WT_Morl}
        (s~\id - \mathcal{K})^{-1}[g](\bb{x})
        =\dfrac{\sqrt{2\pi}}{C_{\Gamma_M}}\int_{\sigma\in \R} \psi_{\sigma,\omega_0}(\bb{x}) \left(\dfrac{1}{s - \dfrac{\imunit \omega_0}{\sigma} } -\dfrac{\kappa}{s}\right)~\dfrac{\d \sigma}{|\sigma|}.
    \end{equation}
    Here $\id$ denotes the identity operator on $C(\Omega)$ and $C_{\Gamma_M}$ is the admissibility constant for the complex Morlet wavelet.
\end{lemma}
\begin{proof}
    Since the Koopman semigroup is a $C_0$-semigroup over $C(\Omega)$, its resolvent $(s~\id-\K)^{-1}[g](\bb{x})$ has the semigroup formulation \cite[Theorem II.1.10]{EngelNagel} 
    \begin{equation}\label{eq: KoopResTD}
        (s~\id-\K)^{-1}[g](\bb{x}) = \int_{t=0}^\infty\K_{t}[g](\bb{x}) \exp(-st) \d t.
    \end{equation}
    Then, by plugging in the characterization from \cref{lem: Kdt_act_Morl} into \cref{eq: KoopResTD} we obtain
    \begin{equation*}
        (s~\id-\K)^{-1}[g](\bb{x}) = \int_{t=0}^\infty\left(\dfrac{\sqrt{2\pi}}{C_{\Gamma_M}}\int_{\sigma\in \R} \psi_{\sigma,\omega_0}(\bb{x})\Bigl(\exp(\imunit \omega_0 t/\sigma) - \kappa\Bigr) \dfrac{\d \sigma}{|\sigma|}\right) \exp(-st) \d t.
    \end{equation*}
    We can rearrange last equality as
    \begin{equation}\label{eq: intKoopRes1}
        (s~\id-\K)^{-1}[g](\bb{x}) = \dfrac{\sqrt{2\pi}}{C_{\Gamma_M}}\int_{t=0}^\infty\int_{\sigma\in \R} \psi_{\sigma,\omega_0}(\bb{x})\Bigl(\exp(\imunit \omega t/\sigma) -\kappa\Bigr)\exp(-st)\dfrac{\d \sigma}{|\sigma|} \d t.
    \end{equation}
    Switching the order of integrals in \cref{eq: intKoopRes1} by Fubini's theorem \cite[Corollary 1.7.23]{Tao} we get
    \begin{equation}\label{eq: intKoopRes2}
        (s~\id-\K)^{-1}[g](\bb{x}) = \dfrac{\sqrt{2\pi}}{C_{\Gamma_M}}\int_{\sigma\in \R}\psi_{\sigma,\omega_0}(\bb{x})\left(\int_{t=0}^\infty \Bigl(\exp(\imunit \omega_0 t/\sigma) - \kappa\Bigr) \exp(-st)\d t\right)\dfrac{\d \sigma}{|\sigma|}.
    \end{equation}
    Note that the inner integral in \cref{eq: intKoopRes2} is the Laplace transform of $\Bigl(\exp(\imunit \omega_0 t/\sigma)-\kappa\Bigr)$ at $s$. For this integral, we have the closed form solution
    \begin{equation}
        \int_{t=0}^\infty \Bigl(\exp(\imunit \omega_0 t/\sigma) - \kappa\Bigr) \exp(-st)\d t = \dfrac{1}{s - \dfrac{\imunit \omega_0}{\sigma}} - \dfrac{\kappa}{s}.
    \end{equation}
    Hence for $\Re(s) > 0$, we can write
    \begin{equation*}
        (s~\id-\K)^{-1}[g](\bb{x}) =\dfrac{\sqrt{2\pi}}{C_{\Gamma_M}}\int_{\sigma\in \R}\psi_{\sigma,\omega_0}(\bb{x})\left(\dfrac{1}{s - \dfrac{\imunit \omega_0}{\sigma} } -\dfrac{\kappa}{s}\right)\dfrac{\d \sigma}{|\sigma|}.
    \end{equation*}
\end{proof}

\cref{thm: Koop_res_Morl} suggests that the resolvent evaluated on $g$ has the integral form \cref{eq: Kres_WT_Morl} for each point $\bb{x}\in \M$. By approximating this integral via quadrature on $\sigma$ we obtain
\begin{equation}\label{eq: res_quad_kappa}
    (s~\id - \K)^{-1}[g](\bb{x}) \approx \sum_{j=1}^J \alpha_j\psi_{\sigma_j}(\bb{x})\left(\dfrac{1}{s - \dfrac{\imunit \omega_0}{\sigma_j} } -\dfrac{\kappa}{s}\right)\dfrac{1}{|\sigma_j|},
\end{equation}
where $\alpha_j\in \C$ are some quadrature weights and $\sigma_j\in \R$ are quadrature nodes. This suggests that this approximation lies in the span of $\psi_{\sigma_j}$. Hence, not only the output map $\K_{\Delta t}[g]$ but also $(s~\id - \K)^{-1}[g]$ is well approximated by this EDMD subspace for a good choice of $\sigma_j$. This then suggests that in addition to the output behavior, the EDMD approximation can recover the resolvent of the Koopman operator as well. 

Moreover, this approximation suggests how to pick good $\sigma_j$ for the case when $s$ lies close to the imaginary axis. Indeed, if we pick a large $\omega_0$, $\kappa$ becomes very small since $\kappa := e^{-\omega_0^2/2}$. Hence at \cref{eq: res_quad_kappa}, the term with $\kappa$ is negligible and we get
\begin{equation}\label{eq: ResQuad}
    (s~\id - \K)^{-1}[g](\bb{x}) \approx \sum_{j=1}^J \alpha_j\psi_{\sigma_j}(\bb{x})\dfrac{1}{s - \dfrac{\imunit \omega_0}{\sigma_j}} \dfrac{1}{|\sigma_j|}.
\end{equation}
The second term in the summand at \cref{eq: ResQuad} is an evaluation of the Cauchy kernel at $s$ and $\dfrac{\imunit \omega_0}{\sigma_j}$. So it decays rapidly when $s$ is away from $\dfrac{\imunit \omega_0}{\sigma_j}$. Assuming $\{\sigma_j\}_{j=1}^J$ are distinct and separated and $s$ is close to $\dfrac{\imunit \omega_0}{\sigma_k}$ for a $k=1,\dots,J$, this yields
\begin{equation}\label{eq: apprKoopEfn}
    (s~\id - \K)^{-1}[g](\bb{x}) \approx M\psi_{\sigma_k}(\bb{x})
\end{equation}
for some $M\in \C$. Hence, we expect the EDMD approximation with these new observables to approximate the spectral content of the Koopman operator well for $s\in \C$ that are close to $\{\imunit \omega_0/\sigma_j\}_{j=1}^J$ up to a complex scalar $M$.

\section{cWDMD: Wavelet-based dynamic mode decomposition with continuous wavelet transform}\label{sect: Implementation}
Until this section, we have established the theoretical foundation of the wavelet-based observables and motivated their suitability for EDMD. Specifically, we have shown that they are eigenfunctions of the Koopman semigroup $(\K_{\Delta t})_{\Delta t \geq 0}$ over the space $C(\Omega)$ where $\Omega$ is a compact, forward-invariant subset of $\M$ in \cref{thm: KoopEfn}. Also in \cref{eq: apprKoopEfn}, we argued that the EDMD subspace with wavelet-based observables $\{\psi_{\sigma_j}\}_{j=1}^J$ will approximate the Koopman resolvent evaluated at $\imunit \omega_0/\sigma_j$ closely using the result in \cref{thm: Koop_res_Morl}. 

In this section, we will propose the cWDMD algorithm -- an EDMD algorithm with wavelet-based observables. Then we will show that the EDMD matrix $\mathbf{K}$ acquired from this cWDMD algorithm yields an accurate approximation of the resolvent evaluated at the output function $g$ for spectral parameter $s\in \C$ near $\imunit \omega_0 / \sigma_j$ for two examples: a linear time-invariant system and the chaotic Lorenz system.

\subsection{Revisiting EDMD}\label{ss: EDMDImpl}
To propose an EDMD algorithm with wavelet-based observables, we first recall the EDMD algorithm for discrete-time systems. In \Cref{ss: KOT} we discussed how EDMD can be understood as a data-driven approximation to the Koopman operator.
Specifically, for the dynamical system~\cref{eq: out_DS}, assume that we have access to state dynamics $\bb{z}_k$, for $k=0,1,\ldots,N$ through a given set of observables $\psi_{\sigma_j}:\M\to\C$ for $j=1,2,\ldots,J$. Define 
$\psi = \bbm\psi_1 & \psi_2 & \cdots & \psi_J \ebm^\top$. Thus, we have access to the data $\psi(\bb{z}_{k})\in \mathbb{C}^{J} \in $ for $k=0,1,\ldots,N$.

Define $\mathcal{F}_J$ as the linear span of these observables. The EDMD process then amounts to solving the least-squares problem \cref{eq: EDMD_min} to obtain a matrix $\hat{\bb{K}}$ that captures the action of the Koopman operator $\mathcal{K}$ as expressed in \cref{eq: EDMD_Koop_approx}. 
Define the data matrices $\Psi,\Psi_+\in \mathbb{C}^{J\times N}$
$$\Psi = \bbm \psi(\bb{z}_0)& \psi(\bb{z}_1)& \cdots & \psi(\bb{z}_{N-1}) \ebm \quad \mbox{and} \quad
\Psi_+ = \bbm \psi(\bb{z}_1)& \psi(\bb{z}_2)& \cdots & \psi(\bb{z}_N) \ebm.
$$
Then, the optimization problem~\cref{eq: EDMD_min} can be equivalently reformulated as the least-squares (LS) problem 
\begin{equation}  \label{eq: EDMDFro}
    \hat{\bb{K}} = \argmin_{\bb{K}\in \C^{J\times J}} \| \Psi_{+} - \bb{K} \Psi \|_F^2.
\end{equation}
In its simplest formulation, the resulting EDMD algorithm can be summarized as in \Cref{alg: EDMD}. 

~

\begin{algorithm}[H]
    \caption{Extended Dynamic Mode Decomposition}\label{alg: EDMD}
    \KwData{$\{\psi(\bb{z}_i)\}_{i=0}^N$}
    \KwResult{$\hat{\bb{K}}\in \mathbb{C}^{J\times J}$: Approximated Koopman operator of the dynamical system}
    $\Psi \gets \bbm \psi(\bb{z}_0)& \psi(\bb{z}_1)& \cdots & \psi(\bb{z}_{N-1}) \ebm \in \mathbb{C}^{J\times N}$\;
    $\Psi_+ \gets \bbm \psi(\bb{z}_1)& \psi(\bb{z}_2)& \cdots & \psi(\bb{z}_N) \ebm \in \mathbb{C}^{J\times N}$\;
    \text{Solve} the LS problem \cref{eq: EDMDFro} \text{to obtain }$\hat{\bb{K}}$ \;
\end{algorithm}

~

There are various ways to solve \cref{eq: EDMDFro}. For our numerical examples, we have implemented the xGEDMD routine proposed in \cite[Algorithm 1]{xGEDMD}.

\subsection{cWDMD algorithm}\label{ss: cWDMD}
Recall the dynamical system 
\begin{equation}
    \begin{aligned}
        \dot{\bb{x}} &= \mathcal{T}(\bb{x})\\
        y &= g(\bb{x})
    \end{aligned}~,\quad \left\{\begin{aligned}
        &\bb{x}\in \mathcal{M} \subseteq \R^n \\ 
        &y \in \R.
    \end{aligned}\right.   \tag{\ref{eq: out_DS}}
\end{equation}
and the wavelet-based observables \cref{eq: wave_obs} given as
\begin{equation*}
    \psi_{\sigma}(\bb{x})=
    \mathcal{W}[y(\cdot,\bb{x})](\sigma,0) = \int_{t\in\R} y(t,\bb{x})\overline{\Gamma_{\sigma,0}(t)} \d t.
\end{equation*}
The proposed algorithm will correspond to applying EDMD with these observables. A sketch of the algorithm is given in \cref{alg: cWDMD}, whose steps we will explain next. \cref{alg: cWDMD} essentially applies CWT to output samples and then solves the least-squares problem \cref{eq: EDMD_min}. Although this resembles the WDMD algorithm \cite{Krishnan_2023}, the important distinction is the use of the continuous version of the wavelet transform with a fixed implementation strategy, rather than the discrete version adopted in WDMD. 
This continuous formulation enables the development of a rigorous theoretical justification of invariance properties that do not hold for the WDMD algorithm. Despite this fundamental analytical difference, we retain the terminology and name the proposed algorithm as Wavelet-based Dynamic Mode Decomposition via Continuous Wavelet Transform and abbreviate it as cWDMD. 

For the dynamical system~\cref{eq: out_DS}, 
\cref{alg: cWDMD} assumes to have the output samples $y(t_i,\bb{x}_{0}^{(k)})$ for some equidistant time samples $t_i = i\Delta t$ for some $\Delta t > 0$, $i =0,\dots,N$ and initial conditions $\{\bb{x}_{0}^{(k)}\}_{k=1}^K$. For any initial condition $\bb{x}_{0}^{(k)}$, we first compute the frequencies $\{\hat{y}(\omega_i,\bb{x}_{0}^{(k)})\}_{i=1}^N$ via the Fast Fourier Transform (FFT) algorithm (Line 2). Then we scale them by $\sigma_j\overline{\hat{\Gamma}(\sigma_j\omega_i)}$ (Line 3) and invert the FFT (Line 4) to get the observable  evaluations. In short, Lines 2-4 apply numerical quadrature to approximate the wavelet coefficients $\mathcal{W}[y(~\cdot,\bb{x}_{0}^{(k)})](\sigma_j,i\Delta t) \approx \K_{i\Delta t}[\psi_{\sigma_j}](\bb{x}_{0}^{(k)})$. Line 5 constructs the evaluations for vector valued function $\psi(\bb{x}_{0}^{(k)})$. Then in Lines 6-7 we build the data matrices for this initial condition $\bb{x}_{0}^{(k)}$ and append it to the full data matrices in Lines 8-9. After processing all the initial conditions, we solve the LS problem \cref{eq: EDMDFro} in Line 10.  

~

\begin{algorithm}[H]
    \LinesNumbered
    \caption{cWDMD: Wavelet-based DMD via CWT}\label{alg: cWDMD}
    \KwData{$\{\bb{y}(t_i,\bb{x}_{0}^{(k)})\}_{i=0}^{N}$ for $k=1,\dots,K$,~$\{\sigma_j\}_{j=1}^J\subset \R_+$,~$\Gamma:\R\to\C$.}
    \KwResult{EDMD matrix $\hat{\bb{K}}\in \mathbb{C}^{J\times J}$.}

    \For{$k = 1,\dots,K$}{%
        $\{(\omega_i,\bb{f}_i^{(k)})\}_{i=0}^{N} \gets \text{FFT}\big(\{\bb{y}(t_i,\bb{x}_{0}^{(k)})\}_{i=0}^{N}\big)$
        \tcp*{Compute $\bb{f}_i^{(k)} = \hat{y}(\omega_i,\bb{x}_{0}^{(k)})$}

        $\alpha_{i,j} \gets \sigma_j\hat{\Gamma}(\sigma_j\omega_i)$ \tcp*{Compute $\sigma_j\hat{\Gamma}(\sigma_j\omega_i)$}

        $w_{i,j}^{(k)} \gets \text{IFFT}\big(\{\alpha_{i,j}\bb{f}_i^{(k)}\}_{i=0}^{N}\big)$
        \tcp*{Compute $\K_{i\Delta t}[\psi_{\sigma_j}]$}

        $\bb{w}_i^{(k)} \gets \bbm w_{i,1}^{(k)} & \cdots & w_{i,J}^{(k)}\ebm^\top$\;

        $\Psi^{(k)} \gets \bbm \bb{w}_0^{(k)} & \bb{w}_1^{(k)} & \cdots & \bb{w}_{N-1}^{(k)} \ebm
        \in \mathbb{R}^{J\times N}$\;

        $\Psi_+^{(k)} \gets \bbm \bb{w}_1^{(k)} & \bb{w}_2^{(k)} & \cdots & \bb{w}_N^{(k)} \ebm
        \in \mathbb{R}^{J\times N}$\;

        $\Psi \gets \bbm \Psi & \Psi^{(k)}\ebm$\;
        $\Psi_+ \gets \bbm \Psi_+ & \Psi_+^{(k)}\ebm$\;
    }
    \text{Solve the LS problem \cref{eq: EDMDFro}} \text{to obtain }$\hat{\bb{K}}$\;
\end{algorithm}

~

In addition to the output data, cWDMD requires choosing $\Gamma$ and the scales $\{\sigma_j\}_{j=1}^J$. For $\Gamma$, we pick $\omega_0 =6$ and let $\Gamma$ be the modulated Gaussian. The choice of scales $\{\sigma_j\}_{j=1}^J$ is fundamental in cWDMD, since these scales determine the spectral parameters $s$ for which one expects a close approximation of the resolvent action $(s~\id - \K)^{-1}[g]$. In our experiments, we set $\sigma_j = 2^{\frac{j}{C}}$ with $j = 1,\dots, J$ and $C>0$. The specific choices of $J$ and $C$ are stated in the numerical examples below.

Another implementation detail, which is not present in \Cref{alg: cWDMD}, is the \emph{realification part}. Note that observables $\psi_{\sigma_j}$ are complex valued and working with complex arithmetic is not ideal in practice. Hence, we split our observables into real and complex parts and identify each of them as a separate observable. So in practice, we have $2J$ observables $\{\Re\{\psi_{\sigma_j}\}\} _{j=1}^J\cup\{\Im\{\psi_{\sigma_j}\}\}_{j=1}^J$. One can show that this is the same as assuming that for every $\psi_{\sigma_j}$, $\psi_{-\sigma_j}$ is also an observable. 

\subsection{Numerical Examples}\label{ss: NumLTI}
We will consider two examples: a two dimensional linear time-invariant system and the Lorenz 63 system with parameters that make the system chaotic. For both cases, we will approximate the Koopman resolvent evaluated at $g$ by the approximation \cref{eq: apprKoopEfn} resulting from \cref{thm: Koop_res_Morl}. 

Experiments reported here have been executed on a machine with the 12th Gen Intel(R) Core(TM) i5-12500H x64 based CPU running at 2.50GHz and equipped with 8.00GB RAM. The computer is running on Windows 11 Home Edition 64-bit operating system with version 24H2 and OS build 26100.3476. MATLAB programming language with version 23.2.0.2485118 (R2023b) is used for both the algorithm implementation and testing.

\subsubsection{Toy Example: A linear dynamical system}
We will start with a linear time-invariant (LTI) system given by
\begin{equation}\label{eq: linEx}
\begin{aligned}
    \dot{\bb{x}} &= \bb{A} \bb{x} \\
                y &= \bb{c}^\top\bb{x}
\end{aligned}
~,\quad \text{where } \bb{A} = \left[\begin{array}{rr} -1 & 500 \\ -500& -1  \end{array}\right],~\bb{c} = \left[\begin{array}{c}
     1 \\ 0 
\end{array}\right],~\bb{x}(0) = \bb{x}_0,~t\geq 0.
\end{equation}
The system~\cref{eq: linEx} is of the form \cref{eq: out_DS}, with $\mathcal{T}(\bb{x}) = \bb{A}\bb{x} $ and $g(\bb{x}) =\bb{c}^\top\bb{x} $. Note that \cref{eq: linEx} is an asymptotically stable system since real part of the eigenvalues of $\bb{A}$ are negative. Hence for any initial condition $\bb{x}_0$, we expect the solution $\bb{x}(t)$ to decay to $0$ as $t\to \infty$. 
We assume that we have no prior information of~\cref{eq: linEx}. So we treat \cref{eq: linEx} as a possibly nonlinear system. We have access to samples of the output trajectory ($y(t)=\bb{c}^\top\bb{x}(t)$ for this example) for different initial conditions. We pick $100$ initial conditions uniformly random from $\mathbb{S}^2\subset\R^2$, the circle with radius $20$, i.e.,
$
    \mathbb{S}^2 := \{\bb{x}\in \R^2 \mid \|\bb{x}\| = 20\}
$
and pick $\Omega\subset \R^2$ to be the disc with radius $20$, i.e.,
$\Omega = \{\bb{x}\in \R^2 \mid \|\bb{x}\|\leq 20\}$.
 For each initial condition, we simulate the system for $T=2$ seconds. We sample these trajectories with time stepping $\Delta t = 0.001$ and apply cWDMD to the resulting data. For the scales $\sigma_j$, we set $C=32$ and $J=288$. Since the output $y$ is scalar-valued, this yields $288$ observables. In addition, upon realification, we include $\psi_{-\sigma_j}$ in the observable set for every $j=1,\dots,J$. Hence, the total number of wavelet-based observables becomes $576$. We also add the mean of the trajectory as an additional observable. So in total we have $577$ observables.

For the LTI system \cref{eq: linEx}, we can analytically write the Koopman resolvent evaluated on $g$  for any $s\in \C$~\cite{Cont_KoopRes}, when it is well-defined, as
\begin{equation}\label{eq: LTIKoopEfn}
    (s~\id - \K)^{-1}[g](\bb{x}) = \bb{c}^\top(s\bb{I} - \bb{A})^{-1}\bb{x}.
\end{equation}
We aim to approximate the maximum Koopman resolvent evaluated at $g$ over the imaginary axis. In our example eigenvalues of $\bb{A}$ are $\lambda_{1,2} = -1 \pm 500 \imunit$. Hence we expect the peak to occur at $s=500$ radians, which corresponds to $\approx 79.58$ Hz. This is indeed what we see if we plot the norm of $(s~\id - \K)^{-1}[g]$ when $s$ ranges over the positive imaginary line, see \Cref{fig: LTIBode}.
\begin{figure}[h]
    \centering
    \includegraphics[width=0.7\linewidth]{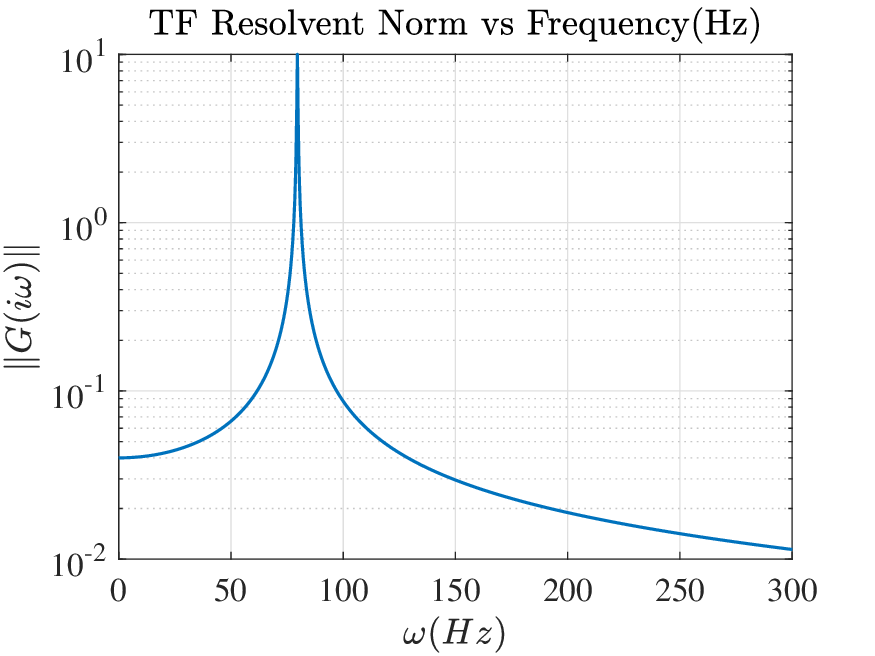}
    \caption{Norm of \cref{eq: LTIKoopEfn} on the positive imaginary axis (in Hz)}
    \label{fig: LTIBode}
\end{figure}

Using \cref{eq: LTIKoopEfn} one can compute the function $(s~\id - \K)^{-1}[g]$ over the state space. The magnitude and argument of this function on the sampled state data is plotted in \Cref{fig: LTIKoopResEv} where the magnitude is normalized to $1$.

\begin{figure}[htbp]
    \centering
    \begin{subfigure}{0.48\linewidth}
        \centering
        \includegraphics[width=\linewidth]{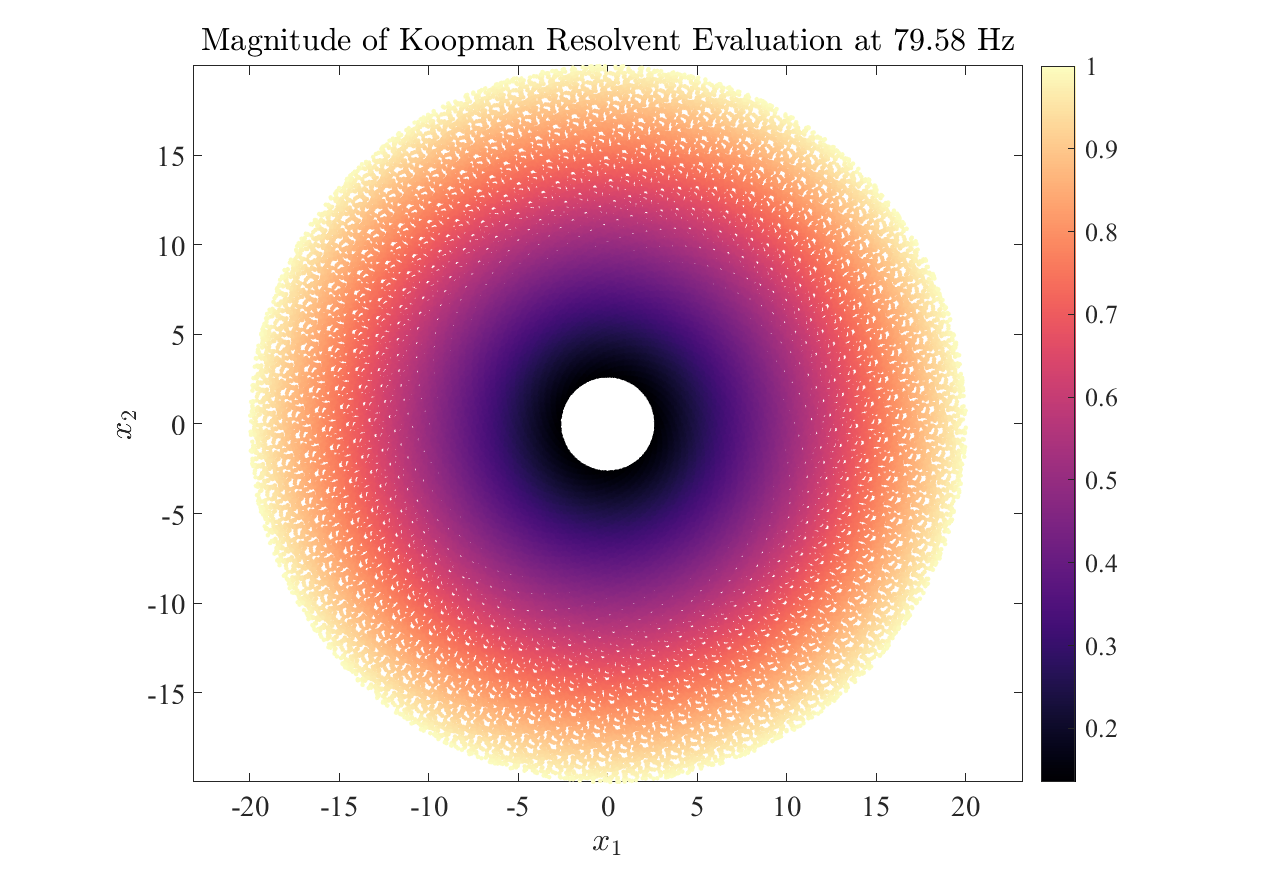}
        \caption{Magnitude}
        \label{fig: LTIKoopResMagEv}
    \end{subfigure}
    \hfill
    \begin{subfigure}{0.48\linewidth}
        \centering
        \includegraphics[width=\linewidth]{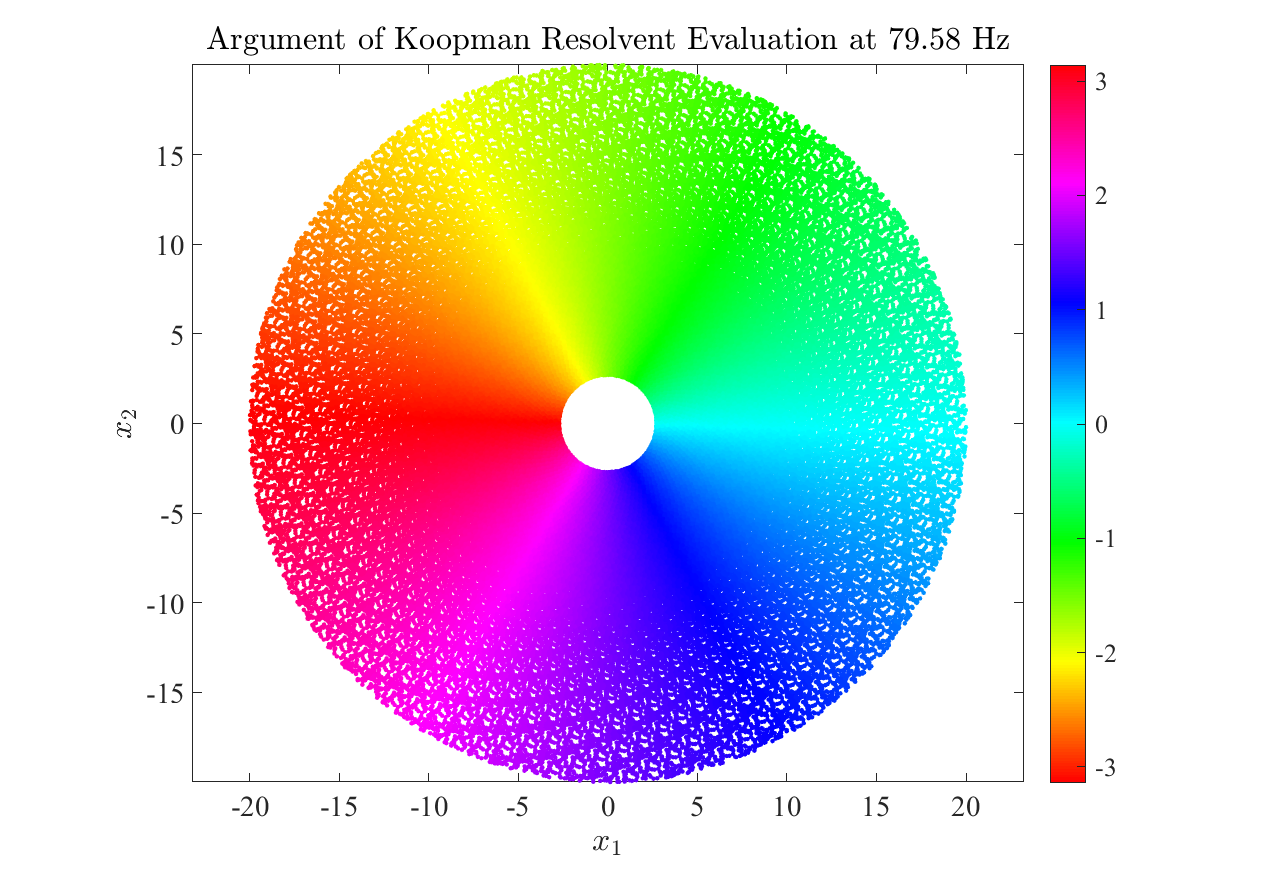}
        \caption{Argument}
        \label{fig: LTIKoopResArgEv}
    \end{subfigure}
    \caption{Koopman Resolvent Evaluation at Peak Frequency}
    \label{fig: LTIKoopResEv}
\end{figure}

\Cref{eq: apprKoopEfn} suggests that we can approximate the resolvent evaluation $(s~\id - \K)^{-1}[g]$ by the wavelet-based observable $\psi_\sigma$ if we pick the scale $\sigma$ such that
$
    \sigma \approx \dfrac{\omega_0}{s}.
$
We will show the approximation at $79.54$ Hz since the peak occurs around that frequency, see \Cref{fig: LTIBode}.

The magnitude and argument of the approximated wavelet-based observable corresponding to the frequency $79.54$ Hz is shown in \Cref{fig: LTIKoopResMagAppr,fig: LTIKoopResArgAppr}. Recall that \cref{thm: KoopEfn} asserts that wavelet-based observables are eigenfunctions of the Koopman operator. Therefore, to obtain an approximation of this observable, we reconstruct the Koopman eigenfunction associated with the corresponding eigenvalue from the EDMD matrix $\bb{K}$ obtained through \cref{alg: cWDMD}. 

In \Cref{fig: LTIKoopResMagAppr,fig: LTIKoopResArgAppr}, both the magnitude and argument are normalized: the magnitude is scaled to have maximum value $1$, while the argument is shifted so that it agrees with the argument of the analytic expression in \cref{eq: LTIKoopEfn} at one predetermined point. This normalization is necessary because \cref{eq: apprKoopEfn} suggests that $\psi_\sigma$ approximates the resolvent evaluation only up to a complex scaling factor $M\in \C$.

\begin{figure}[h]
    \centering
    \includegraphics[width=\linewidth]{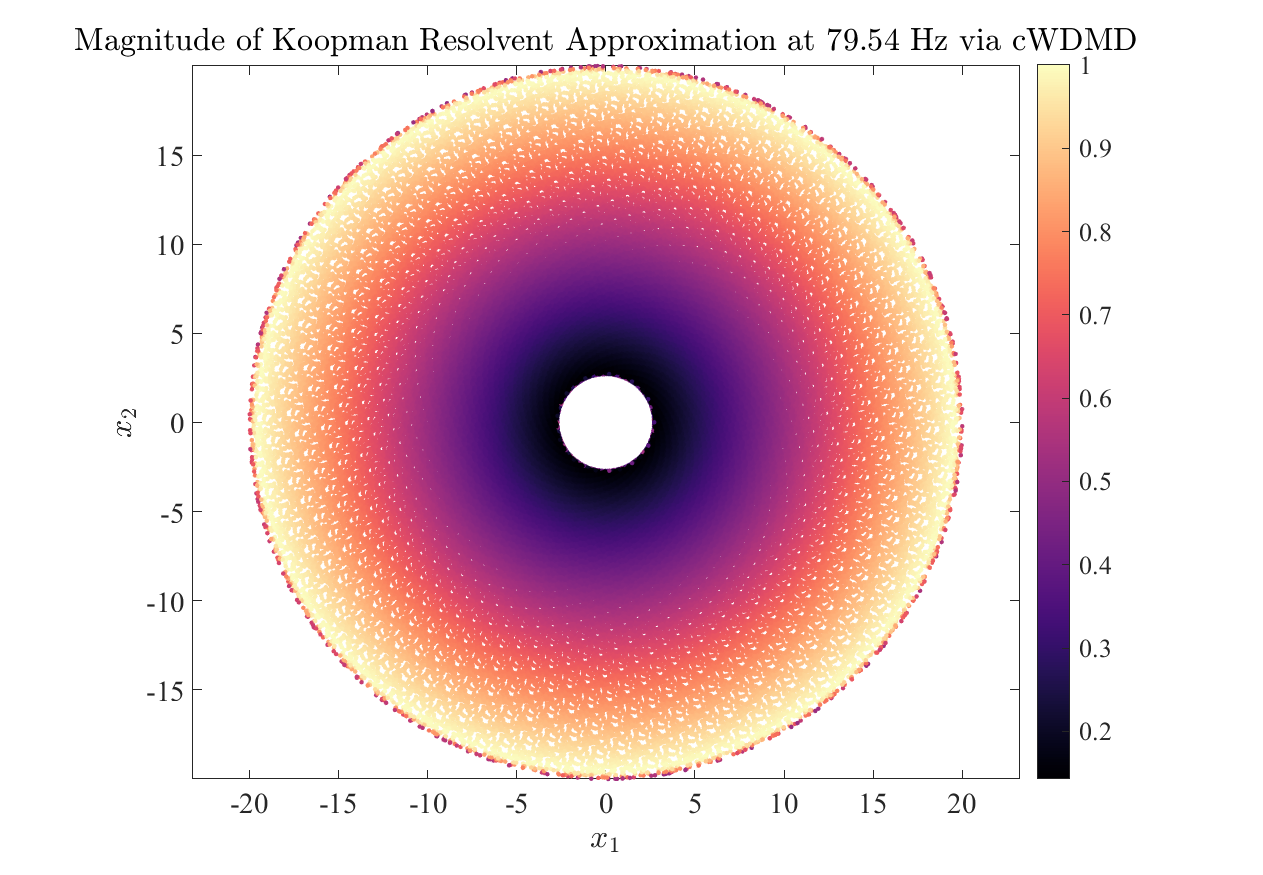}
    \caption{Magnitude of the Koopman Resolvent Approximation at Peak Frequency}
    \label{fig: LTIKoopResMagAppr}
\end{figure}

\begin{figure}[h]
    \centering
    \includegraphics[width=\linewidth]{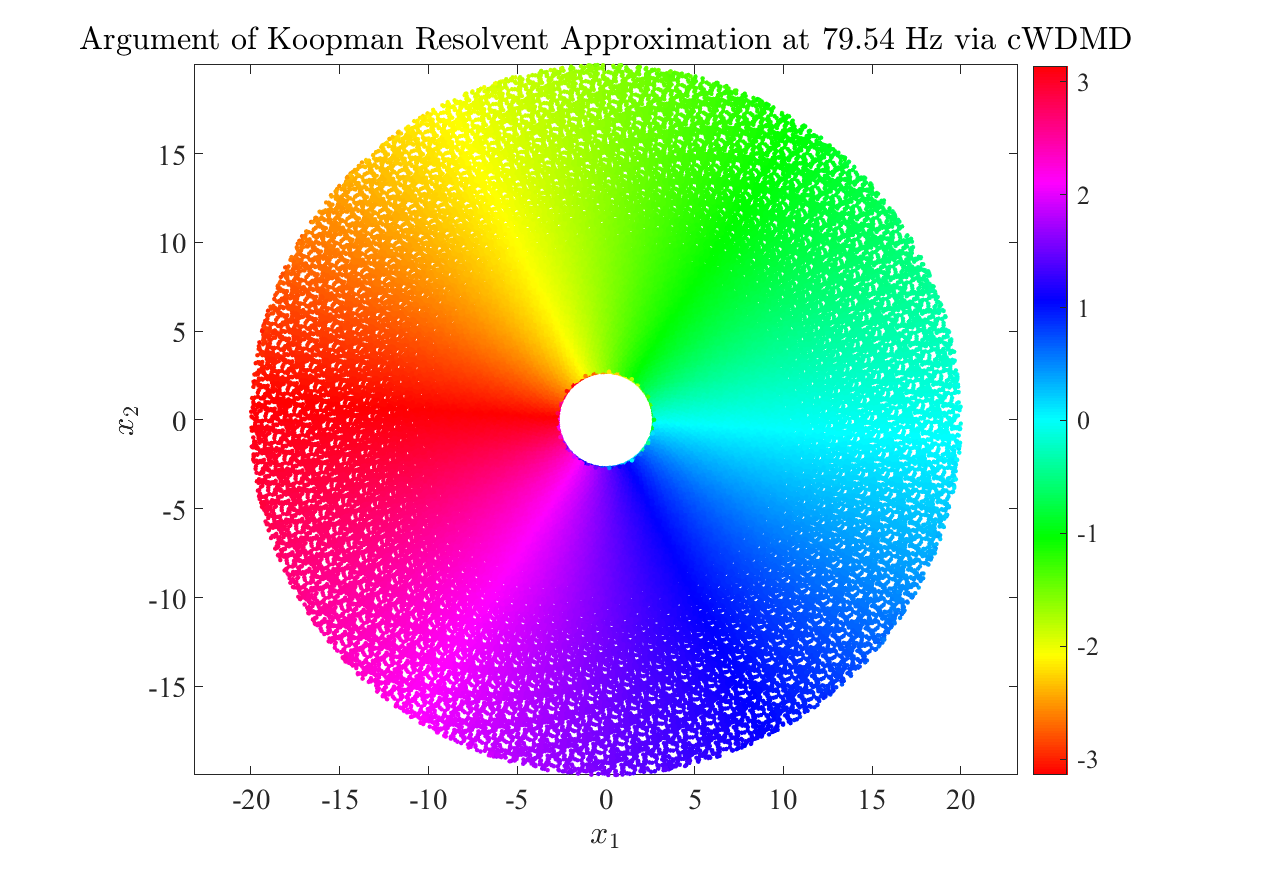}
    \caption{Argument of the Koopman Resolvent Approximation at Peak Frequency}
    \label{fig: LTIKoopResArgAppr}
\end{figure}

Comparing \Cref{fig: LTIKoopResMagAppr,fig: LTIKoopResArgAppr} with \Cref{fig: LTIKoopResEv} we note that the resolvent approximations via cWDMD closely match the analytical formula given in \cref{eq: LTIKoopEfn} except at inner and outer boundaries of the annulus. Hence we claim that this method approximates the resolvents closely. The approximation error at the boundaries stem from the edge effects when computing the wavelet transform. Simply put, since start and end of the trajectories do not contain enough data, the computation error is more prominent at the edges. This can be minimized by increasing the sampling frequency.

\subsubsection{Example: Lorenz system}
For a more complex example, we will consider the Lorenz system. It is described by the system of differential equations
\begin{equation}\label{eq: lorenz}
    \begin{aligned}
    \frac{\d x_1}{\d t} & = \alpha (x_2 - x_1) \\
    \frac{\d x_2}{\d t} & = x_1 (\rho - x_3) - x_2 \\
    \frac{\d x_3}{\d t} & = x_1x_2 - \beta x_3
    \end{aligned},\quad \alpha,\rho,\beta >0.
\end{equation}
Following \cite{riggedEDMD}, we pick our output to be 
\begin{equation*}
    g(\bb{x}) = \textrm{tanh}((x_1x_2 - 5x_3)/10) \quad\textrm{where } \bb{x} = \bbm x_1& x_2 & x_3 \ebm^\top.
\end{equation*}
As in the previous example, \cref{eq: lorenz} is of the form \cref{eq: out_DS}. We use $\alpha = 10,~\rho = 28$ and $\beta = \frac{8}{3}$ to obtain  chaotic behavior. We fix $\Omega\subseteq \R^3$ as the box
$$\Omega = [-20,20]\times [-30,30] \times [0,50]$$
and draw $40$ initial conditions uniformly at random from $\Omega$. We simulate this system on the time interval $[0,100]$, with time step $\Delta t = 0.02$. For scales $\sigma_j$, we pick $C = 20$ and $J = 220$. Considering the realification and the additional mean observable, this amounts to $441$ observables in total. 

Unlike the previous example, we do not have an exact formula for the Koopman resolvent. Hence, we will compare our results with the existing numerical results from literature instead, specifically with the results of  \cite{MezicSpec}. In \cite{LorenzEfn}, authors have numerically verified that the associated Koopman operator has a resonance around $\lambda \approx -0.1 \pm 8\imunit$. This suggests that maximum Koopman resolvent evaluated at $g$ over the imaginary axis should have a peak around $8$ radians per second. This is also confirmed by \cite{MezicSpec} where the authors noticed that the peak is located at $\approx 8.17$ radians per second which corresponds to approximately $ 1.30$ Hz. Moreover the authors have plotted real and imaginary components of the associated resolvent evaluation at \cite[Figure 13]{MezicSpec} with their proposed method. As in the previous example, we will plot our approximation to this using wavelet-based observables, see \Cref{fig: realLorenz,fig: imagLorenz} for real and imaginary parts respectively. 

\begin{figure}[h]
    \centering
    \includegraphics[width=0.7\linewidth]{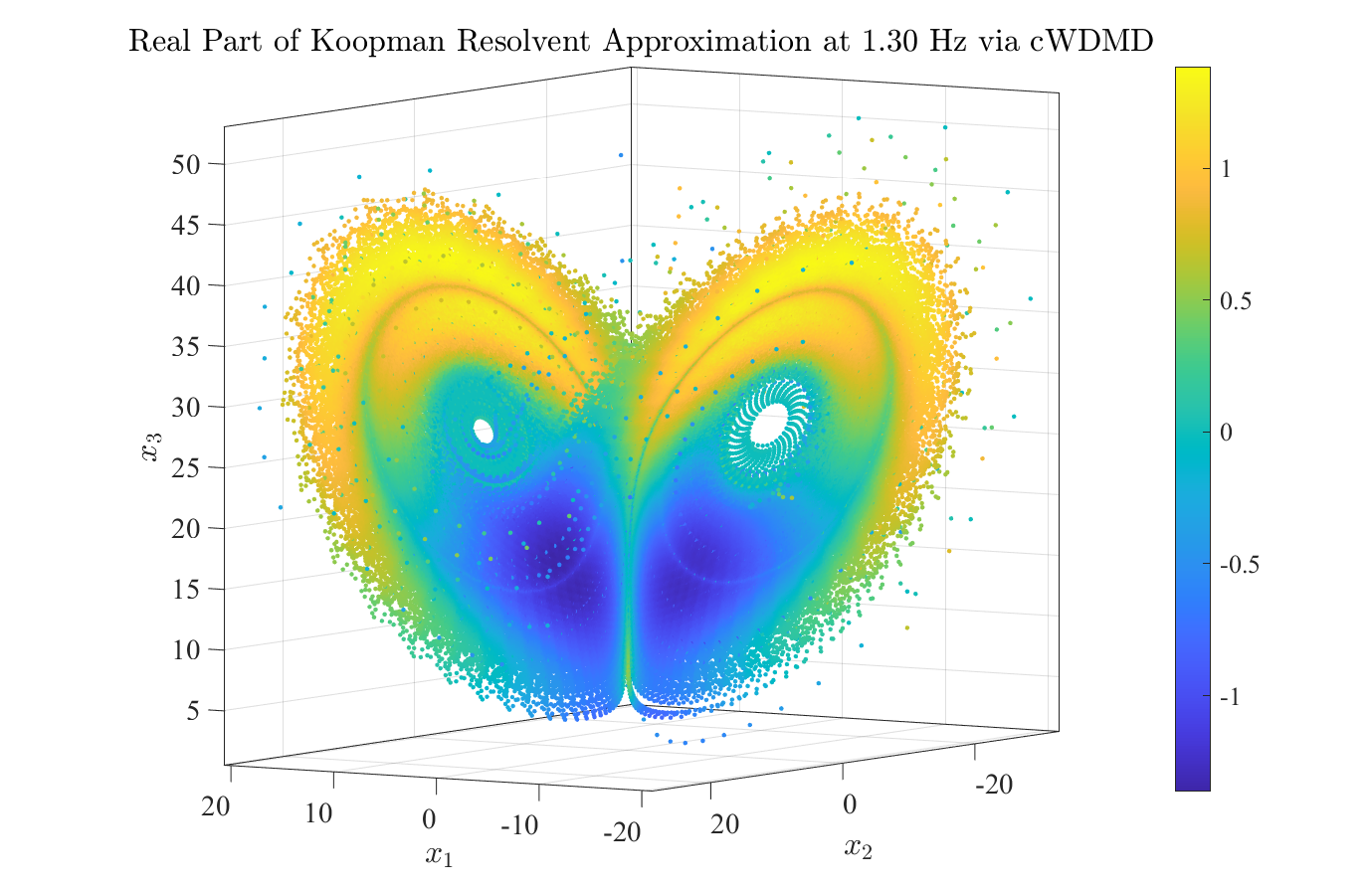}
    \caption{Real Part of the Koopman Resolvent Approximation}
    \label{fig: realLorenz}
\end{figure}

\begin{figure}[h]
    \centering
    \includegraphics[width=0.7\linewidth]{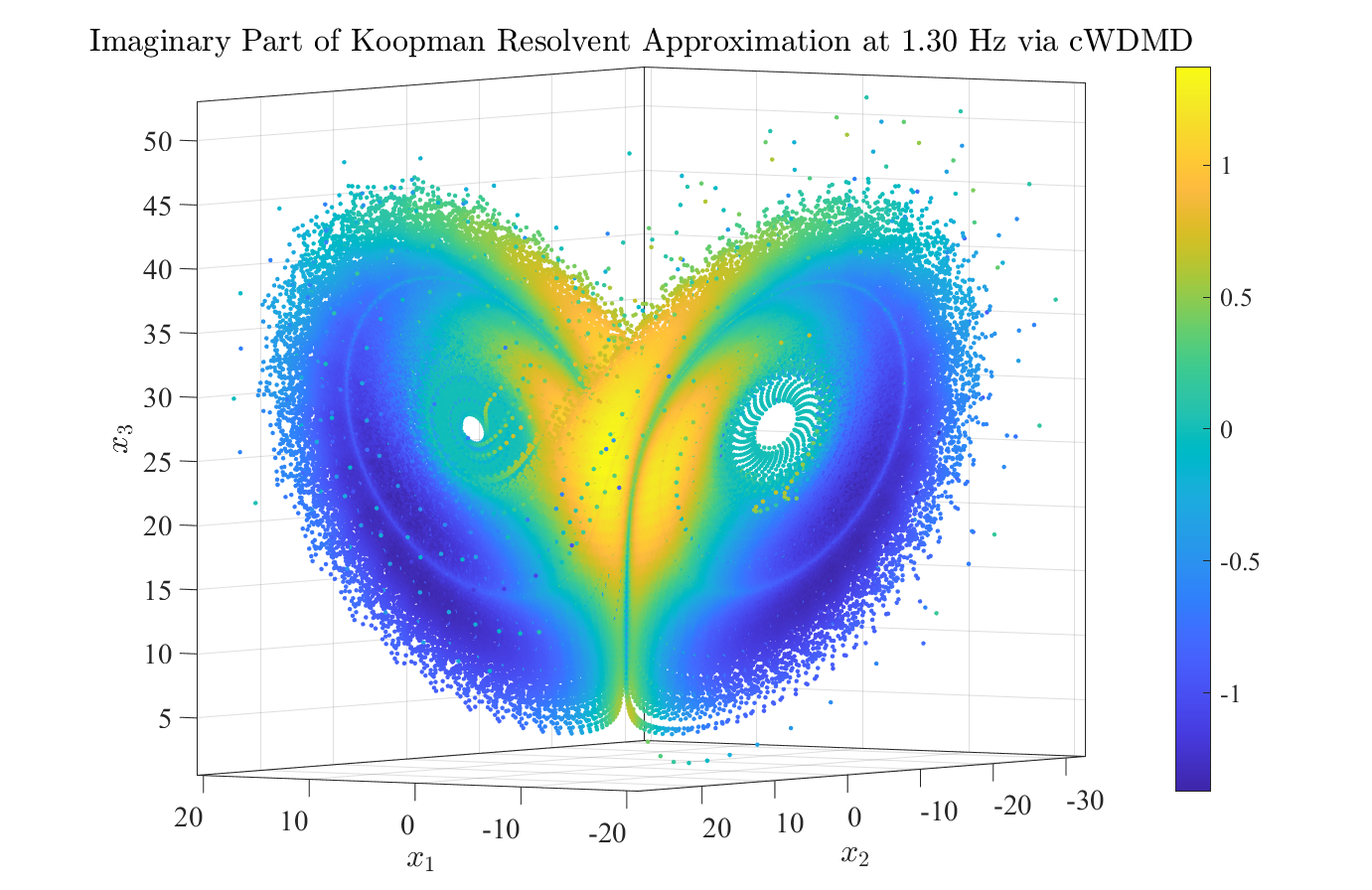}
    \caption{Imaginary Part of the Koopman Resolvent Approximation}
    \label{fig: imagLorenz}
\end{figure}

Note that these figures closely resemble the real and imaginary parts observed in \cite[Figure 13]{MezicSpec}. In short, we have shown that cWDMD eigenfunctions approximates the Koopman resolvent evaluation at $g$ as proposed by \cref{eq: apprKoopEfn} at prescribed frequencies. This holds in practice even for the chaotic systems like Lorenz system with chaotic parameters.

\section{Conclusions}\label{sect: conc}

In this article, we analyzed the Koopman operator using wavelet transform and progressed towards the cWDMD algorithm; a CWT-based EDMD implementation. To do that we first expressed the action of the Koopman semigroup $(\K_{\Delta t})_{\Delta t\geq 0}$ on a function $f\in C(\Omega)$ via the wavelet transform in \cref{thm: Koop_act_wt}. Motivated by this result, we have defined wavelet-based observables in \cref{eq: wave_obs} and shown that they are eigenfunction of the Koopman semigroup at \cref{thm: KoopEfn} when $\Gamma$ is picked as a modulated Gaussian. Moreover, we have provided formulas for recovering the output trajectory and Koopman resolvent from these wavelet-based observables in \cref{lem: Kdt_act_Morl,thm: Koop_res_Morl} respectively. This theoretical reasoning showed us that the wavelet-based observable with parameter $\sigma$ approximates $(s~\id - \K)^{-1}[g]$ for $s = \imunit \omega_0/\sigma$ closely up to a complex scaling. We have used this result to compute the Koopman resolvent of an LTI system and chaotic Lorenz system.


\addcontentsline{toc}{section}{References}
\bibliographystyle{plainurl}
\bibliography{references}

\end{document}